\documentclass[11 pt]{amsart}
\usepackage{latexsym,amscd,amssymb, graphicx, amsthm, bm}  
\usepackage{young}
\usepackage{arydshln}
\usepackage{tikz}

\usepackage[margin=1in]{geometry}

\numberwithin{equation}{section}

\newtheorem{theorem}{Theorem}[section]
\newtheorem{proposition}[theorem]{Proposition}
\newtheorem{corollary}[theorem]{Corollary}
\newtheorem{lemma}[theorem]{Lemma}
\newtheorem{conjecture}[theorem]{Conjecture}

\newtheorem{problem}[theorem]{Problem}
\newtheorem{example}[theorem]{Example}
\newtheorem{remark}[theorem]{Remark}

\newtheorem{defn}[theorem]{Definition}
\theoremstyle{definition}

\newcommand{\maj}{{\mathrm {maj}}}

\newcommand{\sign}{{\mathrm {sign}}}

\newcommand{\Val}{{\mathrm {Val}}}
\newcommand{\Rise}{{\mathrm {Rise}}}
\newcommand{\Stir}{{\mathrm {Stir}}}
\newcommand{\Hilb}{{\mathrm {Hilb}}}
\newcommand{\grFrob}{{\mathrm {grFrob}}}
\newcommand{\des}{{\mathrm {des}}}

\newcommand{\rev}{{\mathrm {rev}}}

\newcommand{\Frob}{{\mathrm {Frob}}}

\newcommand{\ann}{{\mathrm {ann}}}

\newcommand{\symm}{{\mathfrak{S}}}

\newcommand{\wt}{{\mathrm{wt}}}

\newcommand{\CC}{{\mathbb {C}}}
\newcommand{\QQ}{{\mathbb {Q}}}
\newcommand{\ZZ}{{\mathbb {Z}}}

\newcommand{\RR}{{\mathbb{R}}}

\newcommand{\aaa}{{\mathbf {a}}}
\newcommand{\xx}{{\mathbf {x}}}

\newcommand{\yy}{{\mathbf {y}}}

\begin{document}

\title[Vandermondes in superspace]
{Vandermondes in superspace}

\author{Brendon Rhoades}
\address
{Department of Mathematics \newline \indent
University of California, San Diego \newline \indent
La Jolla, CA, 92093, USA}
\email{bprhoades@ucsd.edu}

\author{Andrew Timothy Wilson}
\address
{Department of Mathematics \newline \indent
Portland State University \newline \indent
Portland, OR, 97201, USA}
\email{andwils2@pdx.edu}

\begin{abstract}
Superspace of rank $n$ is a $\QQ$-algebra 
with $n$ commuting generators 
$x_1, \dots, x_n$ and $n$ anticommuting generators $\theta_1, \dots, \theta_n$.
We present an extension of the Vandermonde determinant to superspace which depends
on a sequence $\aaa = (a_1, \dots, a_r)$ of nonnegative integers of length $r \leq n$.
We use  superspace Vandermondes to construct  graded representations of 
the symmetric group.
This construction recovers hook-shaped Tanisaki quotients, the coinvariant ring for the Delta Conjecture
constructed by Haglund, Rhoades, and Shimozono, and a superspace quotient related to positroids
and Chern plethysm constructed by Billey, Rhoades, and Tewari.
We define a notion of partial differentiation with respect to anticommuting variables to construct
doubly graded modules from superspace Vandermondes. 
These doubly graded modules carry a natural ring structure which satisfies a 2-dimensional version of 
Poincar\'e duality.
The application of polarization operators gives rise to other bigraded modules which give a conjectural 
module for the symmetric function $\Delta'_{e_{k-1}} e_n$ appearing in the 
Delta Conjecture of Haglund, Remmel, and Wilson.
\end{abstract}

\maketitle

\section{Introduction}
\label{Introduction}

In this paper we extend the Vandermonde from the classical polynomial ring in $n$ variables to 
a noncommutative deformation of this ring called {\em superspace}.
We use  superspace Vandermondes  to generate 
interesting graded symmetric group modules including
\begin{itemize}
\item
a family $R_{n,k}$ of quotient rings introduced by Haglund, Rhoades, and Shimozono
\cite{HRS} with connections to   the 
cohomology of  Pawlowski-Rhoades moduli spaces of
 spanning line configurations \cite{PR, RW}
(Theorems~\ref{delta-vandermonde-theorem} and \ref{higher-constant-sequences}), 
\item
a class of quotient rings studied by Billey, Rhoades, and Tewari \cite{BRT} related to positroids
and Chern plethysm (Proposition~\ref{w-structure}), 
\item
the Tanisaki quotients $R_{\lambda}$ corresponding to hook-shaped partitions $\lambda \vdash n$
which present the cohomology of the corresponding Springer fiber
$\mathcal{B}_{\lambda}$  \cite{Tanisaki} (Proposition~\ref{zero-tanisaki}) and conjecturally other
Tanisaki quotients $R_{\lambda}$ (Conjecture~\ref{tanisaki-conjecture}), 
\item
a class of {\em doubly} graded modules with a bigraded multiplication
which exhibit a kind of rotational symmetry (Corollary~\ref{duality-corollary})
and a 2-dimensional version of Poincar\'e duality (Corollary~\ref{poincare-corollary}), and
\item
 a class of doubly graded modules whose bigraded Frobenius image is conjecturally given 
by the expression $\Delta'_{e_{k-1}} e_n$ appearing in the
 Delta Conjecture of Haglund, Remmel, and Wilson \cite{HRW} (Conjecture~\ref{double-frobenius}).
\end{itemize}

Let $\QQ[\xx_n] := \QQ[x_1, \dots, x_n]$ be the polynomial ring in $n$ variables equipped with the action
of the symmetric group $\symm_n$ by subscript permutation.
The {\em Vandermonde} 
$\Delta_n$ is an important element of $\QQ[\xx_n]$ with several equivalent definitions.
If we let $\varepsilon_n \in \QQ[\symm_n]$ be the  group algebra element 
\begin{equation}
\varepsilon_n := \sum_{w \in \symm_n} \sign(w) \cdot w,
\end{equation}
we have
\begin{equation}
\label{classical-vandermonde-definition}
\Delta_n := \prod_{1 \leq i < j \leq n} (x_i - x_j) =
\det \begin{pmatrix}
x_1^{n-1} & x_2^{n-1} & \cdots & x_n^{n-1} \\
 & & \vdots & \\
x_1 & x_2 & \cdots & x_n\\
 1 & 1 & \cdots & 1
\end{pmatrix} =
\varepsilon_n \cdot (x_1^{n-1} x_2^{n-2} \cdots x_{n-1}^1 x_n^0).
\end{equation}

For a positive integer $n$, {\em superspace} of rank $n$ is the unital associative $\QQ$-algebra with generators 
$x_1, x_2, \dots, x_n$ and $\theta_1, \theta_2, \dots, \theta_n$ subject to the relations
\begin{equation}
x_i x_j = x_j x_i, \quad x_i \theta_j = \theta_j x_i, \quad \theta_i \theta_j = - \theta_j \theta_i
\end{equation}
for all $1 \leq i, j \leq n$.
Abusing notation, we use $\QQ[\xx_n, \bm{\theta}_n]$ to denote superspace of rank $n$, with the understanding
that the $\theta$-variables anticommute.
The `super'  refers to supersymmetry in physics; the $x$-variables correspond to 
bosons whereas the $\theta$-variables correspond to fermions (see e.g. \cite{PS}).
Extending coefficients to $\RR$, 
superspace is the ring of polynomial-valued differential forms
on Euclidean $n$-space; in this setting the variable $\theta_i$ would be 
more commonly written $dx_i$.
We also have the tensor product model $\QQ[\xx_n, {\bm \theta_n}] = \mathrm{Sym}(V^*) \otimes \bigwedge (V^*)$,
where $V$ is an $n$-dimensional vector space and $V^*$ is its dual space.
Superspace carries a natural bigrading by considering $x$-degree and $\theta$-degree separately.

We endow $\QQ[\xx_n, \bm{\theta}_n]$ with the diagonal action of $\symm_n$:
\begin{equation}
w \cdot x_i := x_{w(i)}, \quad w \cdot \theta_i := \theta_{w(i)} \quad \text{for $w \in \symm_n$ and $1 \leq i \leq n$.}
\end{equation}
This action of $\symm_n$ on $\QQ[\xx_n, \bm{\theta}_n]$ has been used in \cite{BRT, Zabrocki}
to build interesting graded $\symm_n$-modules connected to 
Chern plethysm and delta operators.
We use the last formulation in \eqref{classical-vandermonde-definition}
of the Vandermonde determinant to extend 
Vandermondes to superspace
and construct graded $\symm_n$-modules of our own.
The following superspace elements will be our object of study.

\begin{defn}
\label{supervandermonde-definition}
Let $k, r \geq 0$ with $n = k + r$ and
let $\aaa = (a_1, \dots, a_r)$ be a list of $r$ nonnegative 
integers.  The {\em $\aaa$-superspace Vandermonde} is the following element of 
$\QQ[\xx_n, \bm{\theta}_n]$:
\begin{equation*}
\Delta_n(\aaa) := \varepsilon_n \cdot (x_1^{a_1} x_2^{a_2} \cdots x_r^{a_r} 
x_{r+1}^{k-1} x_{r+2}^{k-2} \cdots x_{n-1}^1 x_n^0 \cdot \theta_1 \theta_2 \cdots \theta_r).
\end{equation*}
\end{defn}

\begin{example}
\label{supervandermonde-example}
 Let $n = 3$.
 Using the anticommutativity of the $\theta$-variables,
\begin{align*}
\Delta_3(1,1) &= 2 x_1 x_2 \theta_1 \theta_2 - 2 x_1 x_3 \theta_1 \theta_3 + 2 x_2 x_3 \theta_2 \theta_3 \\
\Delta_3(2,0) &= x_1^2 \theta_1 \theta_2 + x_2^2 \theta_1 \theta_2 - x_1^2 \theta_1 \theta_3 -
x_3^2 \theta_1 \theta_3 + x_2^2 \theta_2 \theta_3 + x_3^2 \theta_2 \theta_3\\
\Delta_3(1) &= x_1 x_2 \theta_1 - x_1 x_2 \theta_2 - x_1 x_3 \theta_1 - x_2 x_3 \theta_3 
+ x_2 x_3 \theta_2 + x_1 x_3 \theta_3.
\end{align*}
\end{example}

Example~\ref{supervandermonde-example} illustrates that $\Delta_n(\aaa) \in \QQ[\xx_n, \bm{\theta}_n]$
is nonzero even when the sequence $\aaa$ has repeated entries. 
Indeed, the case where $\aaa$ is a constant sequence will be the primary focus of this paper.

Definition~\ref{supervandermonde-definition} specializes to the classical Vandermonde $\Delta_n$
when $\aaa = \varnothing$ is the empty sequence of length zero.  
If $\aaa = (a_1, \dots, a_r)$ is a rearrangement of ${\bf b} = (b_1, \dots, b_n)$,
the anticommutativity of the $\theta$-variables  implies
$\Delta_n(\aaa) = \Delta_n({\bf b})$.

The superpolynomial $\Delta_n(\aaa)$ can be viewed as a
(noncommutative) determinant.  If $A = (A_{i,j})_{1 \leq i, j \leq n}$
is a matrix whose elements lie in $\QQ[\xx_n, \bm{\theta}_n]$, we define
$\det(A) \in \QQ[\xx_n, \bm{\theta}_n]$ by 
\begin{equation}
\det(A) := \sum_{w \in \symm_n} \sign(w) A_{1,w(1)} A_{2, w(2)} \cdots A_{n, w(n)},
\end{equation}
where the terms are multiplied in the specified order.
For $\aaa = (a_1,  \dots,  a_r)$ with $n = k + r$ we have 
\begin{equation}
\Delta_n(\aaa) = 
\det \left( \begin{array}{c c c c}
x_1^{a_1} \theta_1 & x_2^{a_1} \theta_2 & \cdots & x_n^{a_1} \theta_n \\
x_1^{a_2} \theta_1 & x_2^{a_2} \theta_2 & \cdots & x_n ^{a_2} \theta_n \\
 & & \ddots & \\
 x_1^{a_r} \theta_1 & x_2^{a_r} \theta_2 & \cdots & x_n^{a_r} \theta_n  \\
 \\ \hdashline[2pt/2pt] \\
 x_1^{k-1} & x_2^{k-1} & \cdots & x_n^{k-1} \\
 x_1^{k-2} & x_2^{k-2} & \cdots & x_n^{k-2} \\
  & & \ddots & \\
  1 & 1 & \cdots & 1
\end{array} \right).
\end{equation}
The authors are unaware of a superspace extension of the factorization
$\Delta_n = \prod_{1 \leq i < j \leq n} (x_i - x_j)$.

We use an action of partial derivatives on superspace to build  $\symm_n$-modules.
For $1 \leq i \leq n$, let $\partial_i: \QQ[\xx_n, \bm{\theta}_n] \rightarrow \QQ[\xx_n, \bm{\theta}_n]$ be the 
unique linear operator which satisfies 
\begin{equation}
\partial_i(\theta_j) = 0 \quad\text{and}\quad \partial_i(x_j) = \delta_{i,j}
\end{equation}
for all $1 \leq j \leq n$ (where $\delta_{i,j}$ is the Kronecker delta)
together with the Leibniz Rule 
\begin{equation}
\partial_i(fg) = f \partial_i(g) + \partial_i(f) g \quad \text{for all $f, g \in \QQ[\xx_n, \bm{\theta}_n]$.}
\end{equation}
The operator $\partial_i$ is partial differentiation with respect to $x_i$
where the $\theta$-variables are regarded as constants.

\begin{defn}
\label{cone-definition}
Suppose $n = k + r$ for $k, r \geq 0$ and let $\aaa = (a_1, \dots, a_r) \in (\ZZ_{\geq 0})^r$.  
Let $V_n(\aaa)$ be the smallest $\QQ$-linear subspace of $\QQ[\xx_n, \bm{\theta}_n]$ 
containing $\Delta_n(\aaa)$ which is closed under the $n$
partial derivative operators $\partial_1, \dots, \partial_n$.
\end{defn}

Since the superpolynomial $\Delta_n(\aaa)$ is {\em alternating}:
\begin{equation}
w \cdot \Delta_n(\aaa) = \sign(w) \Delta_n(\aaa) \quad \text{for all $w \in \symm_n$}
\end{equation}
the vector space $V_n(\aaa)$
is closed under the action of $\symm_n$.
The space $V_n(\aaa)$ is concentrated in $\theta$-degree $r$ and is a graded vector space with respect
to $x$-degree.
Ignoring the (constant) $\theta$-degree, we regard $V_n(\aaa)$ as a singly graded $\symm_n$-module 
by considering $x$-degree.

The classical Vandermonde $\Delta_n$ gives a model
for the coinvariant ring of $\symm_n$.  For $1 \leq i \leq n$, let $e_d = e_d(\xx_n)$ be the degree
$d$ {\em elementary symmetric polynomial}:
\begin{equation}
e_d := \sum_{1 \leq i_1 < \cdots < i_d \leq n} x_{i_1} \cdots x_{i_d}.
\end{equation}
The {\em invariant ideal} $I_n \subseteq \QQ[\xx_n]$ is the ideal $I_n := \langle e_1, e_2, \dots, e_n \rangle$ 
generated by these polynomials. Equivalently, the ideal $I_n$ is generated by the vector space
$\QQ[\xx_n]^{\symm_n}_+$ of $\symm_n$-invariant polynomials with vanishing constant term.
The {\em coinvariant ring} is the quotient $R_n := \QQ[\xx_n]/I_n$.  The ring $R_n$ has the structure
of a graded $\symm_n$-module.

The coinvariant ring is one of the most important representations in algebraic combinatorics.
Chevalley proved \cite{C} that $R_n \cong \QQ[\symm_n]$ as ungraded $\symm_n$-modules, so that 
$R_n$ gives a graded refinement of the regular representation of $\symm_n$. 
Borel showed \cite{Borel} that $R_n$ presents the cohomology of the variety $\mathcal{F \ell}_n$
of complete flags in $\CC^n$.
If $\aaa = \varnothing$ is the empty sequence so that $\Delta_n(\aaa) = \Delta_n$ is the classical Vandermonde,
we have an isomorphism (see \cite{Bergeron}) of graded $\symm_n$-modules
\begin{equation}
\label{classical-cone-isomorphism}
R_n \cong V_n(\varnothing) = \mathrm{span}_{\QQ} 
\{ \partial_1^{b_1} \cdots \partial_n^{b_n} \Delta_n \,:\,
b_1, \dots, b_n \geq 0 \}.
\end{equation}

The isomorphism \eqref{classical-cone-isomorphism} gives two ways of viewing 
the coinvariant algebra, each with virtues and defects.
The space $R_n = \QQ[\xx_n]/I_n$ has a natural multiplication structure, but as a quotient space,
deciding whether two polynomials $f,g  \in \QQ[\xx_n]$ are equal in $R_n$ can be difficult.
The graded vector
space $V_n(\varnothing)$ is not closed under multiplication, but its elements are honest polynomials (not cosets),
so calculating invariants like dimension is more conceptually straightforward.
In this paper we  use the spaces $V_n(\aaa)$ to extend
\eqref{classical-cone-isomorphism} to a wider class of graded $\symm_n$-modules.
\begin{itemize}
\item If $n = k + r$ and $\aaa = (k-1, \dots, k-1)$ is a length $r$ sequence of $(k-1)$'s, then 
$V_n(\aaa)$ is isomorphic as a graded $\symm_n$-module (up to sign twist and grading reversal) 
to the quotient
$R_{n,k} := \QQ[\xx_n]/I_{n,k}$
where $I_{n,k} \subseteq \QQ[x_1, \dots, x_n]$ is the ideal generated by $x_1^k, x_2^k, \dots, x_n^k$ together
with the top $k$ elementary symmetric polynomials $e_n, e_{n-1}, \dots, e_{n-k+1}$
(Theorem~\ref{delta-vandermonde-theorem}).
The ring $R_{n,k}$ was defined by Haglund, Rhoades, and Shimozono \cite{HRS} in connection with 
{\em Delta Conjecture} \cite{HRW} of Macdonald theory.
\item  If $n = k + r$, $k \leq s$, and $\aaa = (s-1, \dots, s-1)$ is a length $r$ sequence of $(s-1)$'s, then 
$V_n(\aaa)$ is isomorphic (up to sign twist and grading reversal) to a two-parameter family $R_{n,k,s}$ 
of quotient rings defined in \cite{HRS} and further studied from a geometric perspective in \cite{PR}
(Theorem~\ref{higher-constant-sequences}).
\item  If $r \leq n$ and $\aaa = (0, \dots, 0)$ is a length $r$ sequence of zeros, then $V_n(\aaa)$ is isomorphic
(up to sign twist and grading reversal) to the Tanisaki quotient $R_{\lambda}$ corresponding to the hook-shaped
partition $\lambda = (r+1, 1, \dots, 1) \vdash n$
(Proposition~\ref{zero-tanisaki}).
\end{itemize}

The modules $V_n(\aaa)$ of Definition~\ref{cone-definition} are defined using classical partial derivative
operators acting on the commuting $x$-variables.
We introduce the following partial differentiation operators which act on the 
anticommuting $\theta$-variables.  
Given $1 \leq i \leq n$, we define a $\QQ[\xx_n]$-endomorphism of superspace by
\begin{equation}
\partial_i^{\theta}(\theta_{j_1} \cdots \theta_{j_r}) := \begin{cases}
(-1)^{s-1} \theta_{j_1} \cdots \widehat{\theta_{j_s}} \cdots \theta_{j_r} & \text{if $j_s = i$,} \\
0 & \text{if $i \notin \{j_1, \dots, j_r\}$},
\end{cases}
\end{equation}
for any $1 \leq j_1 < \cdots < j_r \leq n$,
where $\widehat{\cdot}$ denotes omission.
The operator $\partial_i^{\theta}$ lowers $\theta$-degree by 1 while leaving $x$-degree unchanged.
We use these operators to build the following class of doubly-graded vector spaces.

\begin{defn}
\label{double-cone-definition}
Suppose $n = k + r$ for $k, r \geq 0$ and let $\aaa = (a_1, \dots, a_r) \in (\ZZ_{\geq 0})^r$.  
Let $W_n(\aaa)$ be the smallest $\QQ$-linear subspace of $\QQ[\xx_n, \bm{\theta}_n]$ 
containing $\Delta_n(\aaa)$ which is closed under the $n$
partial derivative operators $\partial_1, \dots, \partial_n$
as well as the $n$ operators $\partial^{\theta}_1, \dots, \partial^{\theta}_n$.
\end{defn}

The space 
$W_n(\aaa)$  has the structure of a doubly 
graded $\symm_n$-module.
By restricting $W_n(\aaa)$ to the top $\theta$-degree component, we recover the singly
graded module $V_n(\aaa)$.  
In contrast to the spaces $V_n(\aaa)$, there is a natural way to put a ring structure on $W_n(\aaa)$.

It will be shown in Lemma~\ref{basic-partial-lemma}
that the operators $\partial_i$ and $\partial^{\theta}_i$ satisfy the same relations as the generators 
$x_i$ and $\theta_i$ of superspace:
\begin{equation}
\partial_i \partial_j = \partial_j \partial_i, \quad
\partial_i \partial^{\theta}_j = \partial^{\theta}_j \partial_i, \quad
\partial^{\theta}_i \partial^{\theta}_j = - \partial^{\theta}_j \partial^{\theta}_i
\end{equation}
for all $1 \leq i, j \leq n$.
Given $f \in \QQ[\xx_n, {\bm \theta_n}]$, we therefore have a well-defined operator
$\partial(f)$ obtained by replacing each $x_i$ in $f$ with a $\partial_i$ and each
$\theta_i$ in $f$ with a $\partial^{\theta}_i$.
For example, we have 
\begin{equation*}
\partial(x_1^2 \theta_1 \theta_2 - x_1 x_3 \theta_1) = 
\partial_1^2 \partial^{\theta}_1 \partial^{\theta}_2 - \partial_1 \partial_3 \partial^{\theta}_1.
\end{equation*}
We have an action of $\QQ[\xx_n, {\bm \theta}_n]$ on itself given by
$f \cdot g := \partial(f)(g)$.
We use this action to define the following family of bigraded quotient rings.

\begin{defn}
\label{ideal-ring-definition}
Suppose $n = k + r$ for $k, r \geq 0$ and let $\aaa = (a_1, \dots, a_r) \in (\ZZ_{\geq 0})^r$.  
Let $I_n(\aaa) \subseteq \QQ[\xx_n, {\bm \theta_n}]$ be the ideal
\begin{equation}
I_n(\aaa) := \{ f \in \QQ[\xx_n, {\bm \theta_n}] \,:\, f \cdot \Delta_n(\aaa) = 0 \}
\end{equation}
and let
\begin{equation}
R_n(\aaa) := \QQ[\xx_n, {\bm \theta_n}]/I_n(\aaa)
\end{equation}
be the corresponding quotient ring.
\end{defn}

We will show (Corollary~\ref{super-annihilator-twist}) that $R_n(\aaa)$ is isomorphic to $W_n(\aaa)$
as  bigraded $\symm_n$-modules.
The ring $R_n(\aaa)$ enjoys a 2-dimensional kind of duality 
(Theorem~\ref{duality-theorem}, Corollary~\ref{duality-corollary}) which states that
twisting $R_n(\aaa)$ by the sign representation is equivalent to `rotating' its bigrading.
We prove that $R_n(\aaa)$ satisfies a 2-dimensional version of 
Poincar\'e duality (Corollary~\ref{poincare-corollary}) which is implied in the case $\aaa = \varnothing$
by the fact that $\mathcal{F \ell}_n$ is a compact smooth projective complex variety.
We propose the problem of finding a geometric interpretation of the 2-dimensional duality 
satisfied by the rings $R_n(\aaa)$.
We further conjecture a 2-dimensional unimodality property of the bigraded Hilbert series of $R_n(\aaa)$
(Conjecture~\ref{unimodality-conjecture})
which is implied by the Hard Lefschetz Theorem when $\aaa = \varnothing$.

For $k \leq n$, Pawlowski and Rhoades \cite{PR} defined the moduli space $X_{n,k}$ 
of $n$-tuples $(\ell_1, \dots, \ell_n)$ of lines in $\CC^k$ such that $\ell_1 + \cdots + \ell_n = \CC^k$.
This space is homotopy equivalent to $\mathcal{F \ell}_n$ when $k = n$ and is a Zariski open
subset of the $n$-fold product $(\mathbb{P}^{k-1})^n$ of $(k-1)$-dimensional complex projective space
with itself.  Although $X_{n,k}$ is a smooth complex manifold, it is almost never compact and so does
not satisfy the hypotheses of Poincar\'e duality; correspondingly, the Hilbert series of the cohomology ring
 $H^{\bullet}(X_{n,k})$
is not palindromic.
When $\aaa = (k-1, \dots, k-1)$ is a length $n-k$ sequence of $(k-1)$'s, the 
$\theta$-degree zero piece of $R_n(\aaa)$ presents the cohomology $H^{\bullet}(X_{n,k})$. 
The results and conjectures of the previous paragraph suggest that
although $H^{\bullet}(X_{n,k})$ does not satisfy 
desired properties  such as Poincar\'e duality and Hard Lefschetz, 
it is a 1-dimensional slice of a 2-dimensional object that does.

The paper is structured as follows.
In {\bf Section~\ref{Background}}
we give background material related to combinatorics and the representation theory of 
$\symm_n$.
In {\bf Section~\ref{Vandermondes and the Delta Conjecture}}
we calculate the graded isomorphism type of $V_n(\aaa)$ for certain constant
sequences $\aaa$ to give a new model for the coinvariant algebra for the Delta 
Conjecture introduced in \cite{HRS}.
In {\bf Section~\ref{Vandermondes and Other Graded Modules}}
we generalize the results in
Section~\ref{Vandermondes and the Delta Conjecture}
to other constant sequences $\aaa$, giving a Vandermonde model
for the hook-shaped Tanisaki quotients in the process.
We also describe a subspace model for a quotient of superspace introduced in \cite{BRT}
which gives a bigraded refinement of a symmetric group action on positroids.
In {\bf Section~\ref{Antisymmetric differentiation}} we define partial differentiation operators on superspace
with respect to antisymmetric variables, prove the relevant duality result, 
and discuss a possible connection to the superspace coinvariant 
ring.
We close in {\bf Section~\ref{Conclusion}}
with some open problems, including an  extension 
of the module $V_n(\aaa)$ to two sets of commuting variables 
(with one set of skew-commuting variables)
with conjectural doubly graded Frobenius image equal the symmetric function
$\Delta'_{e_{k-1}} e_n$ appearing in the 
Delta Conjecture.

\section{Background}
\label{Background}

\subsection{Combinatorics}
It will often be convenient for us to assert identities up to a nonzero scalar. 
To this end,
suppose $f$ and $g$ are elements of the polynomial ring $\QQ[\xx_n]$ or of superspace
$\QQ[\xx_n, \bm{\theta}_n]$.  
We use the notation 
$f \doteq g$ to indicate that there is a nonzero rational number $a \in \QQ - \{0\}$ such that $f = ag$.

Let $R$ be a ring and let $R[q]$ be the ring of polynomials in $q$ with coefficients in $R$.
Given a polynomial $f = r_d q^d + r_{d-1} q^{d-1} + \cdots + r_1 q + r_0 \in R[q]$ with the $r_i \in R$ and 
$r_d \neq 0$, the {\em $q$-reversal} of $f$ is given by
\begin{equation}
\rev_q f := r_0 q^d + r_1 q^{d-1} \cdots + r_{d-1} q^1 + r_d \in R[q].
\end{equation}

Let $n \geq 0$.
A {\em partition} $\lambda$ of $n$ is a weakly decreasing sequence
$\lambda = (\lambda_1 \geq \cdots \geq \lambda_k)$ of positive integers 
with $\lambda_1 + \cdots + \lambda_k = n$. 
We write $\ell(\lambda) = k$ to indicate the number of parts of $\lambda$ and
$\lambda \vdash  n$ to indicate
that $\lambda$ is a partition of $n$.
For $1 \leq i \leq n$, we write $m_i(\lambda)$ for the multiplicity of $i$ as a part of 
$\lambda$.

We identify a partition $\lambda = (\lambda_1, \dots, \lambda_k)$ with its
{\em Ferrers diagram} consisting of $\lambda_i$ left justified boxes in row $i$.  
The Ferrers diagram of $(3,3,1) \vdash 7$ is shown below.
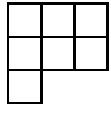
\begin{small}
\begin{center}
\begin{young}
   & & \cr
   & & \cr
   \cr
\end{young}
\end{center}
\end{small}
Given $\lambda \vdash n$, the {\em conjugate} $\lambda'$ is the partition whose Ferrers
diagram is obtained from that of $\lambda$ by reflection across the main diagonal $y = x$.
For example, we have $(3,3,1)' = (3,2,2)$.

We will make use of the following standard $q$-analog notation.
For $n \geq k \geq 0$ we have the {\em $q$-number, $q$-factorial,} and
{\em $q$-binomial coefficient}:
\begin{equation}
[n]_q := 1 + q + \cdots + q^{n-1}, \quad
[n]!_q := [n]_q [n-1]_q \cdots [1]_q, \quad
{n \brack k}_q := \frac{[n]!_q}{[k]!_q \cdot [n-k]!_q}.
\end{equation}
If $k_1 + \cdots + k_r = n$ , we also have the {\em $q$-multinomial coefficient}
\begin{equation}
{n \brack k_1, \dots, k_r}_q := \frac{[n]!_q}{[k_1]!_q \cdots [k_r]!_q}.
\end{equation}

For $n \geq k \geq 0$, let $\Stir(n,k)$ be the {\em (signless) Stirling number of the second kind}
counting $k$-block set partitions of $[n] := \{1, 2, \dots, n\}$.
The {\em $q$-Stirling number} $\Stir_q(n,k)$ is defined by the recursion
\begin{equation}
\Stir_q(n,k) = \Stir_q(n-1,k-1) + [k]_q \cdot \Stir_q(n-1,k)
\end{equation}
together with the initial condition $\Stir_q(0,k) = \delta_{k,0}$ (Kronecker delta).

An {\em ordered set partition} of $[n]$ is a sequence $(B_1, \dots, B_k)$ 
of nonempty subsets of $[n]$ such that we have the disjoint union
$[n] := B_1 \sqcup \cdots \sqcup B_k$. 
The number of $k$-block ordered set partitions of $[n]$ is
$k! \cdot \Stir(n,k)$.

\subsection{Representation theory}
A {\em supermonomial} in $\QQ[\xx_n, {\bm \theta_n}]$ is a product 
$x_1^{a_1} \cdots x_n^{a_n} \theta_{i_1} \cdots \theta_{i_r}$ for some 
$a_1, \dots, a_n \geq 0$ and $1 \leq i_1 < \cdots < i_r \leq n$.
The {\em $x$-degree} of this supermonomial is $a_1 + \cdots + a_n$ 
and the {\em $\theta$-degree} is $r$.

The family of supermonomials forms a basis for $\QQ[\xx_n, {\bm \theta_n}]$;
we call elements of $\QQ[\xx_n, {\bm \theta_n}]$ 
{\em superpolynomials}.  If $f$ is a superpolynomial, the {\em $x$-degree} of $f$
is the largest $x$-degree of the terms appearing in $f$.  We call $f$
{\em $x$-homogeneous} if all of its terms have the same $x$-degree.
The terms {\em $\theta$-degree} and {\em $\theta$-homogeneous} have analogous
meanings. 
If $f$ is simultaneously $x$-homogeneous and $\theta$-homogeneous, we 
call $f$ {\em homogeneous}.

Let $M$ be a $\QQ[\xx_n]$-module.
For $m \in M$, the {\em annihilator} of $m$ is the subset
\begin{equation}
\ann_{\QQ[\xx_n]}(m) := \{ r \in \QQ[\xx_n] \,:\, r \cdot m = 0 \} \subseteq \QQ[\xx_n].
\end{equation}
The subset $\ann_{\QQ[\xx_n]}(m) \subseteq \QQ[\xx_n]$ is an ideal.

Let $V = \bigoplus_{d \geq 0} V_d$ be a graded vector space with each graded piece $V_d$
finite-dimensional.  The {\em Hilbert series} of $V$ is the  power series
\begin{equation}
\Hilb(V;q) := \sum_{d \geq 0} \dim(V_d) q^d.
\end{equation}

Let $\Lambda$ be the ring of symmetric functions over the ground field $\QQ(q,t)$ in an infinite
variable set $\xx = (x_1, x_2, \dots )$.
The ring $\Lambda$ is graded by degree; we let 
$\Lambda_n$ be the homogeneous piece of degree $n$.

For any $\lambda \vdash n$, we have 
the {\em Schur function} $s_{\lambda} = s_{\lambda}(\xx) \in \Lambda_n$.
The family $\{ s_{\lambda} \,:\, \lambda \vdash n \}$ of all such symmetric functions forms a basis 
for $\Lambda_n$.  
The {\em omega involution} is the linear map $\omega: \Lambda \rightarrow \Lambda$
defined on the Schur basis by $\omega(s_{\lambda}) := s_{\lambda'}$.
It can be shown that $\omega$ is a ring homomorphism.

The irreducible representations of $\symm_n$ over the field $\QQ$ are indexed by partitions of $n$.
If $\lambda \vdash n$, let $S^{\lambda}$ be the corresponding irreducible representation of $\symm_n$.
For example, the trivial representation of $\symm_n$ is $S^{(n)}$ and the sign representation
of $\symm_n$ is $S^{(1^n)}$.

The {\em Frobenius map} gives a relationship between the Schur basis and 
the representation theory of $\symm_n$.
Given any finite-dimensional $\symm_n$-module $V$, there are unique multiplicities $m_{\lambda} \geq 0$
such that 
\begin{equation}
V \cong_{\symm_n} \bigoplus_{\lambda \vdash n} m_{\lambda} S^{\lambda}. 
\end{equation}
The {\em Frobenius image} 
$\Frob(V) \in \Lambda_n$ is the symmetric function 
\begin{equation}
\Frob(V) := \sum_{\lambda \vdash n} m_{\lambda} s_{\lambda}.
\end{equation}

If $V$ is a finite-dimensional $\symm_n$-module and $\sign$ denotes the $1$-dimensional sign representation
of $\symm_n$, the tensor product $\sign \otimes V$ is another $\symm_n$-module. The effect of 
tensoring with the sign representation on Frobenius image is the application of the omega involution, that is
\begin{equation}
\Frob(\sign \otimes V) = \omega(\Frob(V)).
\end{equation}

Most of the modules we consider in this paper will be graded. If $V = \bigoplus_{i \geq 0} V_i$ is a graded 
$\symm_n$-module with each piece $V_i$ finite-dimensional, the {\em graded Frobenius image} of $V$
is the series 
\begin{equation}
\grFrob(V; q) := \sum_{i \geq 0} \Frob(V_i) \cdot q^i.
\end{equation}
Similarly, if $V = \bigoplus_{i, j \geq 0} V_{i,j}$ is a bigraded $\symm_n$-module with each bigraded piece
$V_{i,j}$ finite-dimensional, we set
\begin{equation}
\label{bigraded-frobenius-definition}
\grFrob(V; q,t) := \sum_{i, j \geq 0} \Frob(V_{i,j}) \cdot q^i t^j.
\end{equation}
The bigraded Frobenius image \eqref{bigraded-frobenius-definition}
can be extended to define multigraded Frobenius images $\grFrob(V; q_1, q_2, q_3, \dots )$ 
in the obvious way.

We will need the induction product of symmetric group modules.
Let $G$ be a group and let $H$ be a subgroup of $G$.
If $V$ is a representation of $H$, let $V \uparrow_H^G$ be the
induction of $V$ from $H$ to $G$.
If $V$ is a $\symm_n$-module and $W$ is a $\symm_m$-module, the tensor product
$V \otimes W$ is naturally a $\symm_n \times \symm_m$-module.
Viewing $\symm_n \times \symm_m$ as a subgroup of $\symm_{n + m}$ where
$\symm_n$ acts on the first $n$ letters and $\symm_m$ acts on the last $m$ letters,
the {\em induction product} of $V$ and $W$ is the $\symm_{n + m}$-module
\begin{equation}
V \circ W := (V \otimes W) \uparrow_{\symm_n \times \symm_m}^{\symm_{n + m}}.
\end{equation}
The corresponding effect on Frobenius images is
\begin{equation}
\Frob(V \circ W) = \Frob(V) \cdot \Frob(W).
\end{equation}

For $\lambda \vdash n$, let $\widetilde{H}_{\lambda} = \widetilde{H}_{\lambda}(\xx;q,t)
\in \Lambda_n$ be the associated
{\em modified Macdonald symmetric function}.
As with  Schur functions, the set $\{ \widetilde{H}_{\lambda} \,:\, \lambda \vdash n \}$ 
forms a basis for $\Lambda_n$.

Given any symmetric function $F$, the (primed and unprimed) delta operators 
$\Delta_F, \Delta'_F: \Lambda \rightarrow \Lambda$
are the Macdonald eigenoperators defined by
\begin{align}
&\Delta_F: \widetilde{H}_{\lambda} \mapsto F[B_{\lambda}(q,t)] \cdot \widetilde{H}_{\lambda}  \\
&\Delta'_F: \widetilde{H}_{\lambda} \mapsto F[B_{\lambda}(q,t) - 1] \cdot \widetilde{H}_{\lambda} 
\end{align}
The eigenvalue $F[B_{\lambda}(q,t)] \in \QQ(q,t)$ involved in $\Delta_F$ is the plethyistic shorthand
\begin{equation}
F[B_{\lambda}(q,t)] := F( \dots, q^i t^j, \dots ),
\end{equation}
where $(i, j)$ range over all pairs of nonnegative integers such that $i < \lambda_{j+1}$.
The eigenvalue $F[B_{\lambda}(q,t) - 1] \in \QQ(q,t)$ involved in $\Delta'_F$ has the same definiton
as $F[B_{\lambda}(q,t)]$ except that $1 = q^0 t^0$ does not appear as an argument.
\footnote{The Macdonald eigenoperators $\Delta_F$ and $\Delta'_F$ are not to be confused
with the Vandermonde $\Delta_n$ and its superspace extension $\Delta_n(\aaa)$.}
By linearity, the operators $\Delta_F$ and $\Delta'_F$ extend to operators on the full vector space
$\Lambda$ of symmetric functions.

Let $e_n$ be the degree $n$ elementary symmetric function.
For $k \leq n$, the {\em Delta Conjecture} of Haglund, Remmel, and Wilson \cite{HRW} predicts the monomial 
expansion of the symmetric function $\Delta'_{e_{k-1}} e_n$.
It reads
\begin{equation}
\label{delta-conjecture}
\Delta'_{e_{k-1}} e_n = \Rise_{n,k}(\xx;q,t) = \Val_{n,k}(\xx;q,t),
\end{equation}
where $\Rise$ and $\Val$ are certain formal power series involving the combinatorics of lattice paths.
For more details, see \cite{HRW}.

The Delta Conjecture asserts the equality of three formal power series involving the infinite set of variables
$\xx$ together with the two additional parameters $q$ and $t$.  This conjecture remains open, but is known
to be true when one of these parameters is set to zero. 
Combining results of \cite{GHRY, HRW, HRS, Rhoades, WMultiset} we have
\begin{equation}
\label{five-equality}
\Delta'_{e_{k-1}} e_n \mid_{t = 0} = \Rise_{n,k}(\xx;q,0) = \Rise_{n,k}(\xx;0,q) 
= \Val_{n,k}(\xx;q,0) = \Val_{n,k}(\xx;0,q).
\end{equation}
Let $C_{n,k}(\xx;q)$ be the common symmetric function of Equation~\eqref{five-equality}.

For $\lambda \vdash n$,
we will  need the Hall-Littlewood $Q'$-function $Q'_{\lambda}(\xx;q)$.
This may be defined in terms of the modified Macdonald polynomials by
\begin{equation}
Q'_{\lambda}(\xx;q) := \rev_q \widetilde{H}_{\lambda}(\xx;q,0).
\end{equation}
In the special case $\lambda = (1^n)$, the Hall-Littlewood function
gives the graded isomorphism type of the coinvariant ring $R_n$ attached to $\symm_n$,
up to grading reversal:
\begin{equation}
\grFrob(R_n; q) = \rev_q Q'_{(1^n)}(\xx; q).
\end{equation}

\section{Vandermondes and the Delta Conjecture}
\label{Vandermondes and the Delta Conjecture}

\subsection{Vandermondes and annihilators}
In this paper we will study the graded Frobenius images
 $\grFrob(V_n(\aaa); q)$ for various sequences $\aaa \in (\ZZ_{\geq 0})^r$.
 In order to do this, we use the following action of the polynomial ring $\QQ[\xx_n]$ on superspace
 $\QQ[\xx_n, \bm{\theta}_n]$.
 
 Recall from Section~\ref{Introduction} that we have an action of the partial derivative operator 
 $\partial_i$ on $\QQ[\xx_n, \bm{\theta}_n]$ for each $1 \leq i \leq n$.
 Since $\partial_i \partial_j = \partial_j \partial_i$ for all $1 \leq i, j \leq n$,
 given any polynomial $f \in \QQ[\xx_n]$ we may define $\partial(f)$ to be the 
 differential operator on 
 $\QQ[\xx_n, \bm{\theta}_n]$ obtained from $f$ by replacing each $x_i$ by $\partial_i$. 
 The action of $\QQ[\xx_n]$ on $\QQ[\xx_n, \bm{\theta}_n]$ is given by
 \begin{equation}
 f \cdot g := \partial(f)(g) \quad \text{for all $f \in \QQ[\xx_n]$ and $g \in \QQ[\xx_n, \bm{\theta}_n]$}.
 \end{equation}
 This is related to the action of $\symm_n$ in that
 \begin{equation}
 \label{permutation-compatible}
 w \cdot (f \cdot g) = (w \cdot f) \cdot (w \cdot g)
 \end{equation}
 for all $w \in \symm_n$, $f \in \QQ[\xx_n]$, and $g \in \QQ[\xx_n, \bm{\theta}_n]$.
 
 For any $r \leq n$ and any sequence $\aaa \in (\ZZ_{\geq 0})^r$,
the annihilator
 \begin{equation}
 \ann_{\QQ[\xx_n]} \Delta_n(\aaa) = \{ f \in \QQ[\xx_n] \,:\, f \cdot \Delta_n(\aaa) = 0 \}
 \end{equation}
in $\QQ[\xx_n]$ of the $\aaa$-superspace Vandermonde is an ideal in $\QQ[\xx_n]$.
Since $\Delta_n(\aaa)$ is homogeneous in the $x$-variables, the 
annihilator $\ann_{\QQ[\xx_n]} \Delta_n(\aaa)$ is
homogeneous.
Equation~\eqref{permutation-compatible} 
and the fact that $\Delta_n(\aaa)$ is alternating
imply that $\ann_{\QQ[\xx_n]} \Delta_n(\aaa)$is closed under the action
of $\symm_n$.  The quotient $\QQ[\xx_n]/\ann_{\QQ[\xx_n]} \Delta_n(\aaa)$ therefore has the structure of
a graded $\symm_n$-module.
The graded $\symm_n$-modules $V_n(\aaa)$ and $\QQ[\xx_n]/\ann_{\QQ[\xx_n]} \Delta_n(\aaa)$
are related as follows.

\begin{proposition}
\label{annihilator-twist}  
Let $r, k \geq 0$ with $n = r + k$ and let $\aaa \in (\ZZ_{\geq 0})^r$.  
We have
\begin{equation}
\grFrob(\QQ[\xx_n]/\ann_{\QQ[\xx_n]} \Delta_n(\aaa) ; q) = (\rev_q \circ \omega) \grFrob( V_n(\aaa); q).
\end{equation}
\end{proposition}

Proposition~\ref{annihilator-twist} is standard, but we include a proof for completeness.

\begin{proof}
The action of $\QQ[\xx_n]$ on $\QQ[\xx_n, \bm{\theta}_n]$ gives a canonical map 
$\varphi: \QQ[\xx_n]/\ann_{\QQ[\xx_n]} \Delta_n(\aaa) \rightarrow V_n(\aaa)$:
\begin{equation}
\varphi: f  \mapsto f \cdot \Delta_n(\aaa) = \partial(f)(\Delta_n(\aaa)).
\end{equation}
The definitions of $V_n(\aaa)$ and $\ann_{\QQ[\xx_n]} \Delta_n(\aaa)$ guarantee that $\varphi$ is well-defined
and bijective.  Since $\Delta_n(\aaa)$ is an alternant, for $w \in \symm_n$ we have
\begin{equation}
\varphi(w \cdot f) = (w \cdot f) \cdot \Delta_n(\aaa) = \sign(w) (w \cdot f) \cdot (w \cdot \Delta_n(\aaa)) =
\sign(w) w \cdot \varphi(f),
\end{equation}
so that $\varphi$ twists by the sign representation. The degree reversal comes from the fact that $\varphi$
is defined using an action of partial derivatives.
\end{proof}

\subsection{A vanishing lemma}
Proposition~\ref{annihilator-twist} is our basic tool for identifying the graded modules $V_n(\aaa)$.
Our first example is inspired by the Delta Conjecture.

For positive integers $k \leq n$, following \cite{HRS} we define an
ideal $I_{n,k} \subseteq \QQ[\xx_n]$ by
\begin{equation}
I_{n,k} := \langle x_1^k, x_2^k, \dots, x_n^k, e_n, e_{n-1}, \dots, e_{n-k+1} \rangle
\end{equation} 
and let 
\begin{equation}
R_{n,k} := \QQ[\xx_n]/I_{n,k} 
\end{equation}
be the corresponding quotient.
The ring $R_{n,k}$ specializes to the classical coinvariant ring $R_n = \QQ[\xx_n]/I_n$
when $k = n$ and plays the role of
 the coinvariant ring for the Delta Conjecture:
 Haglund, Rhoades, and Shimozono proved \cite{HRS, HRS2} that 
\begin{equation}
\label{r-graded-frobenius}
\grFrob(R_{n,k}; q) = (\rev_q \circ \omega) C_{n,k}(\xx;q) = 
(\rev_q \circ \omega) \Delta'_{e_{k-1}} e_n \mid_{t = 0}.
\end{equation}
On the geometric side,
Pawlowski and Rhoades \cite{PR} showed that $R_{n,k}$ presents the cohomology of the space 
\begin{equation}
\label{x-variety-definition}
X_{n,k} := \{ (\ell_1, \dots, \ell_n) \,:\, \text{$\ell_i$ is a $1$-dimensional subspace of $\CC^k$ and
$\ell_1 + \cdots + \ell_n = \CC^k$} \}.
\end{equation}
of spanning configurations of $n$ lines in $\CC^k$.

Equation~\eqref{r-graded-frobenius} says that 
$R_{n,k}$ has graded Frobenius characteristic equal to the $t = 0$ Delta Conjecture
{\bf upon applying the twist $\rev_q \circ \omega$}.  Our first main result
(Theorem~\ref{delta-vandermonde-theorem} below) uses
superspace Vandermondes to remove this twist.
If $\aaa = (k-1, \dots, k-1)$ is a length $r$ sequence of $(k-1)$'s and $k + r = n$ we will show that
\begin{equation}
\label{first-goal-equation}
\grFrob(V_n(\aaa); q) = C_{n,k}(\xx;q) = \Delta'_{e_{k-1}} e_n \mid_{t = 0}.
\end{equation}

Thanks to Proposition~\ref{annihilator-twist},
Equation~\eqref{r-graded-frobenius}, and the definition of $I_{n,k}$,
Equation~\eqref{first-goal-equation} is equivalent to the assertion
\begin{equation}
\label{ideal-equality-goal}
I_{n,k} = \ann_{\QQ[\xx_n]} \Delta_n(\aaa)
\end{equation}
for the sequence $\aaa = (k-1, \dots, k-1) \in (\ZZ_{\geq 0})^r$. 
Our basic tool in proving \eqref{ideal-equality-goal} is 
the following lemma, which gives elements in $I_n(\aaa)$ for any 
sequence $\aaa \in (\ZZ_{\geq 0})^r$.

\begin{lemma}
\label{sign-reversing-lemma}
Let $k, r$ be nonnegative integers with $k + r = n$ and let $\aaa = (a_1, \dots, a_r) \in (\ZZ_{\geq 0})^r$ 
be an arbitrary length
$r$ sequence of nonnegative integers. The top $k$ elementary symmetric polynomials
$e_n, e_{n-1}, \dots, e_{n-k+1} \in \QQ[\xx_n]$ lie in the annihilator $\ann_{\QQ[\xx_n]} \Delta_n(\aaa)$.
\end{lemma}

\begin{proof}
We need to check that $e_d \cdot \Delta_n(\aaa) = \partial(e_d) \Delta_n(\aaa) = 0$ for 
$n - k < d \leq n$.  We describe a combinatorial procedure for applying the differential operator
$\partial(e_d)$ to $\Delta_n(\aaa)$ for any $1 \leq d \leq n$.

Given $w \in \symm_n$ and $1 \leq i \leq n$, we have the following identity of operators on
$\QQ[\xx_n, {\bm \theta}_n]$:
\begin{equation}
\partial_{w(i)} \cdot w =  w \cdot \partial_i
\end{equation}
For $1 \leq d \leq n$, since $e_d \in \QQ[\xx_n]$ is a symmetric polynomial we have
\begin{equation}
\partial(e_d) \cdot w = w \cdot \partial(e_d)
\end{equation}
which implies
\begin{equation}
\label{commuting-equation}
\partial(e_d) \cdot \varepsilon_n = \varepsilon_n \cdot \partial(e_d)
\end{equation}
and therefore
\begin{align}
\partial(e_d)  \cdot \Delta_n(\aaa) &= 
\partial(e_d) \cdot \varepsilon_n \cdot  (x_1^{a_1} \cdots x_r^{a_r}  x_{r+1}^{k-1} \cdots x_{n-1}^1 x_n^0 
 \cdot \theta_1 \cdots \theta_r)  \\ &=
\varepsilon_n \cdot  \partial(e_d) \cdot  (x_1^{a_1} \cdots x_r^{a_r}  x_{r+1}^{k-1} \cdots x_{n-1}^1 x_n^0 
\cdot \theta_1 \cdots \theta_r).
\end{align}
We describe the application of $\partial(e_d)$ to
$ (x_1^{a_1} \cdots x_r^{a_r}  x_{r+1}^{k-1} \cdots x_{n-1}^1 x_n^0  \cdot \theta_1 \cdots \theta_r)$
combinatorially.

The supermonomial $x_1^{a_1}  \cdots x_r^{a_r}  x_{r+1}^{k-1} \cdots x_{n-1}^1 x_n^0 \cdot \theta_1 \cdots \theta_r$
is modeled by a diagram with $n$ columns of boxes. 
For $1 \leq i \leq r$, the $i^{th}$ column (from left to right) contains a box with a $\theta$ at the bottom,
with $a_i$ empty boxes on top.
For $r+1 \leq i \leq n$, the $i^{th}$ column consists of $n-i$ empty boxes.
The case $n = 7, r = 3, \aaa = (4,4,1)$ is shown below.
We refer to this diagram as the {\em $\aaa$-staircase}.

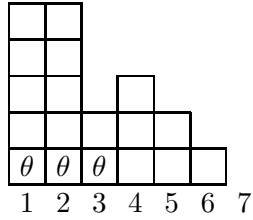
\begin{center}
\begin{young}
 & &, &, &, &, \cr
 & &,  &, &, &, \cr
 & &, & &, & ,\cr
 & & & & &, \cr
$\theta$ & $\theta$ & $\theta$ & & & \cr
,1 & ,2 & ,3 & ,4 & ,5 & ,6 &,7
\end{young}
\end{center}

For $d \geq 0$, a {\em $d$-dotted $\aaa$-staircase} is an $\aaa$-staircase in which $d$ of the boxes
are marked with a $\bullet$, with no two marked boxes in the same column.
An example with $d = 4$ is shown below.

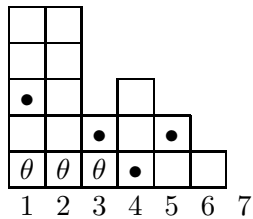
\begin{center}
\begin{young}
 & &, &, &, &, \cr
 & &,  &, &, &, \cr
 $\bullet$ & &, & &, & ,\cr
 & & $\bullet$ &  & $\bullet$&, \cr
$\theta$ & $\theta$ & $\theta$ &  $\bullet$ & & \cr
,1 & ,2 & ,3 & ,4 & ,5 & ,6 &,7
\end{young}
\end{center}

Let $\sigma$ be a $d$-dotted $\aaa$-staircase.  The {\em weight} $\wt(\sigma) \in \QQ[\xx_n, {\bm \theta}_n]$
of $\sigma$ is the  supermonomial
$x_1^{b_1} \cdots x_n^{b_n} \cdot \theta_1 \cdots \theta_r$, where
$b_i$ is the number of empty boxes in column $i$.
In the above example, we have
$\wt(\sigma) = x_1^3 x_2^4 x_4^2 x_5 x_6 \cdot \theta_1 \theta_2 \theta_3$.
It should be clear that 
\begin{equation}
\label{vanishing-sum}
\partial(e_d) \cdot x_1^{a_1}  \cdots x_r^{a_r}  x_{r+1}^{k-1} \cdots x_{n-1}^1 x_n^0 \cdot \theta_1 \cdots \theta_r =
\sum_{\sigma} \wt(\sigma),
\end{equation}
where the sum is over all $d$-dotted $\aaa$-staircases $\sigma$.

In order to calculate $\partial(e_d) \cdot \Delta_n(\aaa)$, we apply $\varepsilon_n$ to both sides of 
Equation~\eqref{vanishing-sum}. By Equation~\eqref{commuting-equation}, this yields
\begin{equation}
e_d \cdot \Delta_n(\aaa) = \partial(e_d) \cdot \Delta_n(\aaa) = \sum_{\sigma} \varepsilon_n \cdot \wt(\sigma).
\end{equation}
Let $\sigma$ be a $d$-dotted $\aaa$-staircase.  If any column of $\sigma$ contains a $\bullet$ but no
$\theta$, there will be two $\theta$-free columns of $\sigma$ with the same number of empty boxes
so that $\varepsilon_n \cdot \wt(\sigma) = 0$.
If $d > n - k$, any $d$-dotted $\aaa$-staircase $\sigma$ has a $\bullet$
in a $\theta$-free column so that $e_d \cdot \Delta_n(\aaa) = 0$.
\end{proof}

\begin{remark}
If we let $X_{n,k}$ be the variety \eqref{x-variety-definition} 
of spanning configurations of $n$ lines $(\ell_1, \dots, \ell_n)$ in $\CC^k$ and let 
$\ell_i^* \twoheadrightarrow X_{n,k}$
be the $i^{th}$ tautological line bundle for $1 \leq i \leq n$, we can identify the variable $x_i$
with the Chern class $x_i := c_1(\ell_i^*) \in H^2(X_{n,k})$.  The Whitney Sum Formula
can be used (see \cite{PR}) to deduce that the top $k$ elementary symmetric polynomials 
$e_n, e_{n-1}, \dots, e_{n-k+1}$ in the $x_i$ vanish in $H^{\bullet}(X_{n,k}; \QQ)$.
Since we have the identification $H^{\bullet}(X_{n,k}; \QQ) = R_{n,k}$ (see \cite{PR}),
this gives geometric intuition for why $e_n, e_{n-1}, \dots, e_{n-k+1}$ `should' lie in the ideal $I_{n,k}$.

Assuming Equation~\eqref{r-graded-frobenius}, Lemma~\ref{sign-reversing-lemma}
gives algebraic and combinatorial
intuition coming from superspace
 for why the elementary symmetric polynomials $e_n, e_{n-1}, \dots, e_{n-k+1}$
`should' lie in the ideal $I_{n,k}$ whose corresponding quotient models the Delta Conjecture at $t = 0$.
\end{remark}

\subsection{A Vandermonde model for $C_{n,k}$}
Our goal is the equality of ideals \eqref{ideal-equality-goal}.
Lemma~\ref{sign-reversing-lemma} gives one of the containments right away.

\begin{lemma}
\label{upper-bound-lemma}
Let $k, r$ be nonnegative integers with $k + r = n$.    Let $\aaa = (k-1, \dots, k-1) \in (\ZZ_{\geq 0})^r$.
Then $I_{n,k} \subseteq \ann_{\QQ[\xx_n]} \Delta_n(\aaa)$.
\end{lemma}

\begin{proof}
It suffices to show that the generators of the ideal $I_{n,k} \subseteq \QQ[\xx_n]$ annihilate the 
$\aaa$-superspace Vandermonde
\begin{equation}
\Delta_n(\aaa) = \varepsilon_n \cdot (x_1^{k-1}  \cdots x_r^{k-1} x_{r+1}^{k-1} \cdots x_{n-1}^1 x_n^0
\cdot \theta_1 \cdots \theta_r).
\end{equation}
Since no $x$-variable appearing in $\Delta_n(\aaa)$ has exponent $\geq k$, we see immediately that
\begin{equation}
x_i^k \cdot \Delta_n(\aaa) = \partial_i^k \Delta_n(\aaa) = 0
\end{equation}
for all $1 \leq i \leq n$.
The remaining generators of $I_{n,k}$ are handled by Lemma~\ref{sign-reversing-lemma}.
\end{proof}

Lemma~\ref{upper-bound-lemma} proves that the module $V_n(\aaa)$ is not too large; it yields a surjection of 
vector spaces
 $R_{n,k} = \QQ[\xx_n]/I_{n,k} \twoheadrightarrow \QQ[\xx_n]/\ann_{\QQ[\xx_n]} \Delta_n(\aaa) \cong  V_n(\aaa)$.  
Our next task is to show that $V_n(\aaa)$ is not too small. To do this, we  need some facts about $R_{n,k}$.

 Haglund, Rhoades, and Shimozono \cite{HRS} proved that $\dim(R_{n,k}) = k! \cdot \Stir(n,k)$, the number 
of $k$-block ordered set partitions of $[n]$.
There are a number of vector space bases of $R_{n,k}$ which are indexed by ordered set partitions
\cite{HRS, PR}: we describe the {\em substaircase monomial basis} here.

Recall that a {\em shuffle} of two sequences $(a_1, \dots, a_r)$ and $(b_1, \dots, b_s)$ is an interleaving
$(c_1, \dots, c_{r+s})$ which preserves the relative orders of the $a$'s and the $b$'s.
If $k + r = n$, and  
{\em $(n,k)$-staircase} is a shuffle of the sequences $(k-1, k-1, \dots, k-1)$ ($r$ times) and 
$(k-1, k-2, \dots, 1, 0)$.
For example, the $(5,3)$-staircases are the shuffles of $(2,2)$ and $(2,1,0)$:
\begin{equation*}
(2,2,2,1,0), \, (2, 2, 1, 2, 0), \, (2, 2, 1, 0, 2), \,  (2, 1, 2, 2, 0), \,  (2, 1, 2, 0, 2), \text{ and } (2,1,0,2,2).
\end{equation*}

A sequence $(c_1, \dots, c_n)$ of nonnegative integers is called {\em $(n,k)$-substairase} if it is
componentwise $\leq$ at lease one $(n,k)$-staircase.
For example, the sequence $(2,0,2,1,0)$ is $(5,3)$-substaircase since we have the 
componentwise inequality $(2,0,2,1,0) \leq (2,2,1,0,2)$.
It is shown in \cite[Thm. 4.13]{HRS} that 
\begin{equation}
\{ x_1^{c_1} \cdots x_n^{c_n} \,:\, (c_1, \dots, c_n) \text{ is $(n,k)$-substaircase} \}
\end{equation}
descends to a vector space basis of $R_{n,k}$.
\footnote{Strictly speaking, \cite[Thm. 4.13]{HRS} states that  the set
$\{ x_n^{c_1} \cdots x_1^{c_n} \,:\, (c_1, \dots, c_n) \text{ is $(n,k)$-substaircase} \}$
of `reversed' monomials descends to a basis of $R_{n,k}$, but since $R_{n,k}$ is
an $\symm_n$-module this is equivalent.}
In particular, 
\begin{equation}
\label{substaircase-count}
\text{there are  $k! \cdot \Stir(n,k)$ sequences $(c_1, \dots, c_n)$
which are $(n,k)$-substaircase.}
\end{equation}

The proof of \eqref{substaircase-count} in \cite{HRS} was recursive and rather involved.
A bijective proof of \eqref{substaircase-count} 
involving an extension of Lehmer code from permutations to
ordered set partitions was given in \cite{RW}.
We will use  substaircase monomials to show that $\dim V_n(\aaa)$ is not too small.

\begin{lemma}
\label{lower-bound-lemma}
Let $k, r \geq 0$ with $k + r = n$ and let $\aaa = (k-1, \dots, k-1) \in (\ZZ_{\geq 0})^r$.
We have $\dim V_n(\aaa) \geq k! \cdot \Stir(n,k)$.
\end{lemma}

\begin{proof}
It is enough to exhibit $k! \cdot \Stir(n,k)$ linearly independent elements of the vector space $V_n(\aaa)$.
Thanks to \eqref{substaircase-count} it is enough to show that 
\begin{equation}
\{ \partial_1^{c_1} \cdots \partial_n^{c_n} \Delta_n(\aaa) \,:\, (c_1, \dots, c_n) \text{ is $(n,k)$-substaircase} \}
\subseteq V_n(\aaa)
\end{equation}
is linearly independent.  We begin with the following seemingly weaker claim.

\noindent
{\bf Claim:} {\em  
 The family of ${n-1 \choose r}$ superpolynomials
\begin{equation}
\label{theta-polynomial-set}
\left\{ \theta_{i_1} \theta_{i_2} \cdots \theta_{i_r} + 
\sum_{j = 1}^r (-1)^j \theta_{1} \theta_{i_1} \cdots \widehat{\theta_{i_j}} \cdots \theta_{i_r} \,:\, 
\begin{array}{c}
\{i_1 < \cdots < i_r \} \text{ is an } \\  \text{ $r$-element subset of $\{2, 3, \dots, n \}$} 
\end{array} \right\}
\end{equation}
is linearly independent. }

To see why the Claim is true,
observe that 
$\theta_{i_1} \theta_{i_2} \cdots \theta_{i_r} $ only appears in the element 
corresponding to $\{i_1 < \cdots < i_r\}$. 
This completes the proof of the Claim.

Let us see how the Claim proves the lemma. Suppose there were numbers $\gamma_{(c_1, \dots, c_n)} \in \QQ$
not all zero so that 
\begin{equation}
\label{false-identity}
\sum_{(c_1, \dots, c_n)} \gamma_{(c_1, \dots, c_n)} \partial_1^{c_1} \cdots \partial_n^{c_n} \Delta_n(\aaa) = 0,
\end{equation}
where the sum is over all $(n,k)$-substaircases $(c_1, \dots, c_n)$.
Choose an $(n,k)$-substaircase $(d_1, \dots, d_n)$ so that
\begin{enumerate}
\item we have $\gamma_{(d_1, \dots, d_n)} \neq 0$, and
\item subject to (1) the number $d_1 + \cdots + d_n$ is minimal, and
\item subject to (1) and (2) the sequence $(d_1, \dots, d_n)$ is lexicographically least.
\end{enumerate}

Let $(d'_1, \dots, d'_n)$ be an $(n,k)$-staircase such that $d_i \leq d'_i$ for all $i$.
Let $1 \leq i_1 < \cdots < i_r$ be the indices such that $d_1 = d_{i_1} = \cdots = d_{i_r} = k-1$.
Write $p_i := d'_i - d_i$ and consider
applying the operator $\partial_1^{p_1} \cdots \partial_n^{p_n}$
to both sides of Equation~\eqref{false-identity}.

The application of $\partial_1^{p_1} \cdots \partial_n^{p_n}$ to the term 
$\gamma_{(c_1, \dots, c_n)} \partial_1^{c_1} \cdots \partial_n^{c_n} \Delta_n(\aaa)$ in 
Equation~\eqref{false-identity} has the following effect.
\begin{itemize}
\item If $(c_1 + p_1, \dots, c_n + p_n)$ is not a rearrangement of 
$(k-1, \dots, k-1, k-2, \dots, 1, 0)$ then
$\gamma_{(c_1, \dots, c_n)} \partial_1^{c_1 + p_1} \cdots \partial_n^{c_n + p_n} \Delta_n(\aaa) = 0$.
\item If $c_1 + p_1 < k - 1$ the lexicographical minimality of $(d_1, \dots, d_n)$ forces
$\gamma_{(c_1, \dots, c_n)} = 0$.
\item Otherwise, let $1 < s_1 < \cdots < s_r$ be the unique indices such that 
\begin{equation*}
c_1 + p_1 = c_{s_1} + p_{s_1} = \cdots = c_{s_r} + p_{s_r} = k - 1. 
\end{equation*}
We have
\begin{equation}
\partial_1^{c_1 + p_1} \cdots \partial_n^{c_n + p_n} \gamma_{(c_1, \dots, c_n)}
 \Delta_n(\aaa)  \doteq \gamma_{(c_1, \dots, c_n)} \left[
\theta_{s_1} \theta_{s_2} \cdots \theta_{s_r} + 
\sum_{j = 1}^r (-1)^j \theta_{1} \theta_{s_1} \cdots \widehat{\theta_{s_j}} \cdots \theta_{s_r} \right].
\end{equation}
\end{itemize}
Observe that the superpolynomial in the final bullet point is a superpolynomial appearing in the Claim.
Furthermore, if $(s_1, \dots, s_r) = (i_1, \dots, i_r)$, the lexicographical finality of
$(d_1, \dots, d_n)$ forces $(c_1, \dots, c_n) = (d_1, \dots, d_n)$.  Our Claim
gives the contradiction $\gamma_{(d_1, \dots, d_n)} = 0$.
\end{proof}

By combining Lemmas~\ref{upper-bound-lemma} and \ref{lower-bound-lemma},
we  prove Equation~\eqref{r-graded-frobenius} and 
obtain our new model for the Delta  coinvariants.

\begin{theorem}
\label{delta-vandermonde-theorem}
Let $k, r$ be nonnegative integers with $k + r = n$.  Let $\aaa = (k-1, k-1, \dots, k-1)$ where there are $r$ copies
of $k-1$.  We have
\begin{equation}
\grFrob(V_n(\aaa); q) = C_{n,k}(\xx;q) = \Delta'_{e_{k-1}} e_n \mid_{t = 0}.
\end{equation}
\end{theorem}

\begin{proof}
By Proposition~\ref{annihilator-twist}, Equation~\eqref{r-graded-frobenius}, and
Lemma~\ref{upper-bound-lemma}, it is enough to show that 
\begin{equation}
\dim V_n(\aaa) \geq \dim \QQ[\xx_n]/I_{n,k}  = \dim R_{n,k} = k! \cdot \Stir(n,k).
\end{equation}
This is a consequence of Lemma~\ref{lower-bound-lemma}.
\end{proof}

\begin{remark}
The irreducible representation $S^{(n-1,1)}$ corresponding to the partition
$(n-1,1) \vdash n$
is the $(n-1)$-dimensional
{\em reflection representation} of $\symm_n$.
Explicitly, this representation is obtained by taking the quotient of the action of $\symm_n$
on $\QQ^n$ by coordinate permutation by the line of constant vectors.

The span  of the  ${n-1 \choose r}$ polynomials 
in the $\theta$-variables described
in \eqref{theta-polynomial-set}
in the proof of Lemma~\ref{lower-bound-lemma} is closed under the action of $\symm_n$.  
This span is 
isomorphic to the exterior power
 $\wedge^{r} S^{(n-1,1)}$ as an $\symm_n$-module. 
If $\lambda = (n-r, 1, \dots, 1) \vdash n$ (where there are $r$ copies of $1$), it is well-known that
$\wedge^{r} S^{(n-1,1)} \cong S^{\lambda}$.
Lemma~\ref{lower-bound-lemma} and Theorem~\ref{delta-vandermonde-theorem}
therefore explain the presence of hook-shaped Schur functions as the coefficient of 
$q^0 t^0$ in the 
Schur expansion of $\Delta'_{e_{k-1}} e_n$.
\end{remark}

The symmetric function $\grFrob(V_n(\aaa); q)$ 
of Theorem~\ref{delta-vandermonde-theorem}
can also be expressed in the Schur basis.
Applying \cite{WMultiset}[Thm. 5.0.1] (with $m = 0$ after taking the coefficient of $u^{n-k}$) 
and \cite[Cor. 6.13]{HRS} we see that
\begin{equation}
\grFrob(V_n(\aaa); q) = 
\sum_{T \in \mathrm{SYT}(n)} 
q^{\maj(T) + {n-k \choose 2} - (n-k) \cdot \des(T)} {\des(T) \brack n-k}_q s_{\mathrm{shape}(T)}.
\end{equation}
Here $\mathrm{SYT}(n)$ is the set of standard Young tableaux with $n$ boxes,
$\maj(T)$ is the major index of $T$, $\des(T)$ is the number of descents in $T$,
and $\mathrm{shape}(T) \vdash n$ is the shape of $T$;
see \cite{HRS} or \cite{WMultiset} for definitions of these terms.
We can describe the Hilbert series of the module in Theorem~\ref{delta-vandermonde-theorem}
in terms of $q$-Stirling numbers.

\begin{corollary}
\label{delta-vandermonde-hilbert}
Let $k, r$ be nonnegative integers with $k + r = n$.  Let $\aaa = (k-1, k-1, \dots, k-1)$ where there are $r$ copies
of $k-1$. 
We have
\begin{equation}
\Hilb(V_n(\aaa); q) = [k]!_q \cdot \Stir_q(n,k).
\end{equation}
\end{corollary}

\begin{proof}
The asserted formula is the $q$-reversal of the formula for  $\Hilb(R_{n,k}; q)$ 
given in \cite{HRS}.
\end{proof}

\section{Vandermondes and Other Graded Modules}
\label{Vandermondes and Other Graded Modules}

In this section we extend Theorem~\ref{delta-vandermonde-theorem}
to calculate $\grFrob(V_n(\aaa); q)$ for other constant vectors $\aaa$.
The first result involves uniformly increasing the entries of $\aaa$.

\subsection{Vandermondes and the quotient ring $R_{n,k,s}$}
Let $k, s,$ and $n$ be nonnegative
integers with $k \leq s$. We define 
the ideal $I_{n,k,s} \subseteq \QQ[\xx_n]$ by 
\begin{equation}
I_{n,k,s} := \langle x_1^s, x_2^s, \dots, x_n^s, e_n, e_{n-1}, \dots, e_{n-k+1} \rangle
\end{equation}
and let $R_{n,k,s} := \QQ[\xx_n]/I_{n,k,s}$ be the corresponding quotient ring.
When $k = s$ we have $I_{n,k,k} = I_{n,k}$ and $R_{n,k,k} = R_{n,k}$.
The rings $R_{n,k,s}$ are graded $\symm_n$-modules which were used in \cite{HRS}
to inductively understand the rings $R_{n,k}$. 
Pawlowski and Rhoades proved \cite{PR} that $R_{n,k,s}$ presents the cohomology of a 
certain space $X_{n,k,s}$ of line configurations.

We extend $(n,k)$-staircases to include the parameter $s$ as follows. 
 An {\em $(n,k,s)$-staircase} is a shuffle of the sequences
 $(s-1, s-1, \dots, s-1)$ (with $n-k$ copies of $s$) and 
 $(k-1, \dots, 1, 0)$.
 For example, the $(4,2,6)$-staircases are the shuffles of $(5,5)$ and $(1,0)$:
 \begin{equation*}
 (5,5,1,0), \, (5,1,5,0), \, (5,1,0,5), \, (1,5,5,0), \, (1,5,0,5), \text{ and } (1,0,5,5).
 \end{equation*}
 A sequence $(c_1, \dots, c_n)$ is {\em $(n,k,s)$-substaircase} if 
 it is componentwise $\leq$ at least one $(n,k,s)$-staircase.
The $(n,k,s)$-substaircase sequences parameterize a monomial basis of $R_{n,k,s}$.
 
 \begin{proposition}
 \label{three-parameter-basis}
 Let $k, s,$ and $n$ be nonnegative integers with $k \leq s$.
 The set of monomials
 \begin{equation*}
  \{ x_1^{c_1} \cdots x_n^{c_n} \,:\, \text{$(c_1, \dots, c_n)$ is an $(n,k,s)$-substaircase} \}
  \end{equation*}
 descends to a basis of $R_{n,k,s}$.
 \end{proposition}
 
 \begin{proof}
 In the case $k \leq s \leq n$,
Haglund, Rhoades, and Shimozono computed  \cite[Sec. 6]{HRS}  the standard monomial
 basis of $R_{n,k,s}$ in terms of `$(n,k,s)$-nonskip monomials'. 
The arguments of \cite[Sec. 6]{HRS} go through without change to the 
 case  $n < s$.
 We show that this monomial basis coincides with the set of $(n,k,s)$-substaircase monomials.
 
 Let $S = \{i_1 < \cdots < i_s\} \subseteq [n]$.  The {\em skip sequence} $\gamma(S) = (\gamma_1, \dots, \gamma_n)$
 corresponding to $S$ is defined by
 \begin{equation}
 \gamma_i = \begin{cases}
 i - j + 1 & \text{if $i = i_j \in S$,} \\
 0 & \text{if $i \notin S$.}
 \end{cases}
 \end{equation}
 We also let $\gamma(S)^* := (\gamma_n, \dots, \gamma_1)$ be the reverse of the sequence $\gamma(S)$.
 A sequence $(c_1, \dots, c_n)$ of nonnegative integers is {\em $(n,k,s)$-nonskip} if
 \begin{itemize}
 \item $c_i < s$ for all $1 \leq i \leq n$ and
 \item the coordinatewise inequality $\gamma(S)^* \leq (c_1, \dots, c_n)$ does {\em not} hold for any $S \subseteq [n]$
 with $|S| = n-k+1$.
 \end{itemize}
 The arguments
 of \cite[Sec. 6]{HRS} show that  the set 
 $\{ x_1^{c_1} \cdots x_n^{c_n} \,:\, (c_1, \dots, c)$ is $(n,k,s)$-nonskip$\}$
 descends to a basis of $R_{n,k,s}$. 
 The proposition therefore reduces to the following
 
 \noindent
 {\bf Claim:}
{\em Let $(c_1, \dots, c_n)$ be a sequence of nonnegative integers. Then $(c_1, \dots, c_n)$ is 
$(n,k,s)$-nonskip if and only if $(c_1, \dots, c_n)$ is $(n,k,s)$-substaircase.}

When $k = s$ this claim is proven in \cite{HRS}, so we assume that $k < s$.
The reverse implication reduces to showing that any $(n,k,s)$-staircase is $(n,k,s)$-nonskip, which 
we leave to the reader.
For the forward implication, let $(c_1, \dots, c_n)$ be an $(n,k,s)$-nonskip sequence.
We produce an $(n,k,s)$-staircase $(b_1, \dots, b_n)$ such that we have the 
componentwise inequality $(c_1, \dots, c_n) \leq (b_1, \dots, b_n)$.

Since we are assuming $k < s$, an $(n,k,s)$-staircase $(b_1, \dots, b_n)$ is determined by the 
set 
\begin{equation*}
T := \{ 1 \leq i \leq n \,:\, b_i < k \} = \{ t_1 < t_2 < \cdots < t_k \}
\end{equation*}
of positions of entries $< k$.  We describe how to form $T$ from $(c_1, \dots, c_n)$.

We claim that there exists $1 \leq t_k \leq n$ such that $c_{t_k} < 1$.
If not, we would have the componentwise inequality 
$(1, 1, \dots, 1) \leq (c_1, c_2, \dots, c_n)$. If $S = \{1, 2, \dots, n-k+1\}$,
we would have $\gamma(S)^* \leq (1, 1, \dots, 1) \leq (c_1, c_2, \dots, c_n)$, contradicting
the assumption that $(c_1, c_2, \dots, c_n)$ is $(n,k,s)$-nonskip.
Let $1 \leq t_k \leq n$ be maximal such that $c_{t_k} < 1$.  

With $t_k$ as in the last paragraph, we claim that there exists $1 \leq t_{k-1} < t_k$ 
with $c_{t_k} < 2$. 
If not, we would have the componentwise inequality
\begin{equation*}
(2, 2, \dots, 2, 0, 1, 1, \dots, 1) \leq (c_1, c_2, \dots, c_n)
\end{equation*}
where the $0$ is in position $t_k$.  If we take $S \subseteq [n]$
to be 
\begin{equation*}
S = \{1, 2, \dots, n - t_k, n - t_k + 2, n - t_k + 3, \dots, n-k+2 \}, 
\end{equation*}
we would have 
$\gamma(S)^* \leq (c_1, c_2, \dots, c_n)$, contradicting
the assumption that $(c_1, c_2, \dots, c_n)$ is $(n,k,s)$-nonskip.
Let $1 \leq t_{k-1} < t_k$ be maximal such that $c_{t_{k-1}} < 2$.

Given $t_{k-1} < t_k$ as above, we claim that there exists $1 \leq t_{k-2} < t_{k-1}$ 
with $c_{t_{k-2}} < 3$.  If not, the componentwise inequality
\begin{equation*}
(3, 3, \dots, 3, 0, 2, 2, \dots, 2, 0, 1, 1, \dots, 1) \leq (c_1, c_2, \dots, c_n)
\end{equation*}
with the $0$'s in positions $t_{k-1}$ and $t_k$ would contradict
$(c_1, \dots, c_n)$ being $(n,k,s)$-nonskip. 
Choose $1 \leq t_{k-2} < t_{k-1}$ minimal such that $c_{t_{k-2}} < 3$.

Since $(c_1, \dots, c_n)$ is $(n,k,s)$-nonskip, we can iterate this procedure to
obtain a $k$-element subset $T = \{t_1 < \cdots < t_k \} \subseteq [n]$.
Let $(b_1, \dots, b_n)$ be the unique $(n,k,s)$-staircase whose entries which are $< k$ are 
in the positions indexed by $T$. 
By the construction of $T$ we have $(c_1, \dots, c_n) \leq (b_1, \dots, b_n)$ so that 
$(c_1, \dots, c_n)$ is $(n,k,s)$-substaircase.
 \end{proof}

 Proposition~\ref{three-parameter-basis} can be used to extend
 Theorem~\ref{delta-vandermonde-theorem} to constant vectors with
 larger entries.

\begin{theorem}
\label{higher-constant-sequences}
Let $k, s,$ and $n$ with $k \leq s$ be nonnegative integers and let $r = n - k$.  Let 
$\aaa = (s-1, s-1, \dots, s-1)$ be the constant vector of length $r$ with entries $s-1$.
We have 
\begin{equation}
\grFrob(V_n(\aaa); q) = (\rev_q \circ \omega) \grFrob(R_{n,k,s}; q).
\end{equation}
\end{theorem}

\begin{proof}
When $k = s$, this is Theorem~\ref{delta-vandermonde-theorem} so we assume $s < k$.

It is enough to demonstrate the equality of ideals
$I_{n,k,s} = \ann_{\QQ[\xx_n]} \Delta_n(\aaa)$.  
The containment $I_{n,k,s} \subseteq \ann_{\QQ[\xx_n]} \Delta_n(\aaa)$ 
follows from Lemma~\ref{sign-reversing-lemma}
and the fact that no $x$-variable in $\Delta_n(\aaa)$ has exponent $\geq s$. The desired equality of ideals
will follow if we can show
\begin{equation}
\dim V_n(\aaa) = \dim(\QQ[\xx_n]/\ann_{\QQ[\xx_n]} \Delta_n(\aaa)) 
\geq \dim(\QQ[\xx_n]/I_{n,k,s}) = \dim R_{n,k,s}.
\end{equation}
By 
Proposition~\ref{three-parameter-basis}, we know that $\dim(R_{n,k,s})$ equals the number of 
$(n,k,s)$-staircases. It is therefore enough to prove the following 

\noindent
{\bf Claim:}
{\em The subset
\begin{equation}
\label{larger-independent-set}
\{ \partial_1^{c_1} \cdots \partial_n^{c_n} \Delta_n(\aaa) \,:\, (c_1, \dots, c_n) \text{ is $(n,k,s)$-substaircase} \}
\subseteq V_n(\aaa)
\end{equation}
is linearly independent.}

Since $k < s$, the Claim is an easier version of Lemma~\ref{lower-bound-lemma}.
The set of ${n \choose r}$ supermonomials
\begin{equation}
\label{smaller-independent-set}
\{ \theta_{i_1}  \cdots \theta_{i_r} \,:\, 1 \leq i_1 < \cdots < i_r \leq n \}
\end{equation}
is certainly linearly independent.  
If the Claim were false, there would be scalars $\gamma_{(c_1, \dots, c_n)}$ not all zero so that
\begin{equation}
\label{higher-false-relation}
\sum_{(c_1, \dots, c_n)} \gamma_{(c_1, \dots, c_n)} \partial_1^{c_1} \cdots \partial_n^{c_n} \Delta_n(\aaa) = 0,
\end{equation}
where the sum is over all $(n,k,s)$-substaircase sequences $(c_1, \dots, c_n)$.

Let $(d_1, \dots, d_n)$ be the unique $(n,k,s)$-substaircase such that 
\begin{enumerate}
\item we have $\gamma_{(d_1, \dots, d_n)} \neq 0$,
\item subject to (1) the number $d_1 + \cdots + d_n$ is minimal, and
\item subject to (1) and (2) the sequence $(d_1, \dots, d_n)$ is lexicographically least.
\end{enumerate}
Let $(d'_1, \dots, d'_n)$ be a $(n,k,s)$-staircase such that $d_i \leq d'_i$ for all $i$ and set $p_i := d'_i - d_i$.
Let $1 \leq i_1 < \cdots < i_r \leq n$ be the indices such that $d'_{i_1} = \cdots = d'_{i_r} = s-1$.

Consider applying the operator $\partial_1^{p_1} \cdots \partial_n^{p_n}$ to both sides of
Equation~\eqref{higher-false-relation}. This operator has the following effect on 
 $\gamma_{(c_1,\dots,c_n)} \partial_1^{c_1} \cdots \partial_n^{c_n} \Delta_n(\aaa)$.
\begin{itemize}
\item If $(c_1 + p_1, \dots, c_n + p_n)$ is not a rearrangement of $(s-1, \dots, s-1, k-1, k-2, \dots, 1, 0)$
then 
$\gamma_{(c_1, \dots, c_n)} \partial_1^{c_1 + p_1} \cdots \partial_n^{c_n + p_n} \Delta_n(\aaa) = 0$.
\item If $(c_1 + p_1, \dots, c_n + p_n)$ is  a rearrangement of $(s-1, \dots, s-1, k-1, k-2, \dots, 1, 0)$
let $1 \leq t_1 < \cdots < t_r \leq n$ be the indices with $c_{t_1} + p_{t_1} = \cdots = c_{t_r} + p_{t_r} = s-1$. 
We have 
\begin{equation*}
\partial_1^{c_1 + p_1} \cdots \partial_n^{c_n + p_n} \Delta_n(\aaa) \doteq \theta_{t_1} \cdots \theta_{t_r}.
\end{equation*}
\end{itemize}
If $(c_1 + p_1, \dots, c_n + p_n)$ is as in the second bullet point and we have $(t_1, \dots, t_r) = (i_1, \dots, i_r)$,
the lexicographical minimality of $(d_1, \dots, d_n)$ forces 
$(c_1, \dots, c_n) = (d_1, \dots, d_n)$.
The linear independence of \eqref{smaller-independent-set} implies that
$\gamma_{(d_1, \dots, d_n)} = 0$, which is a contradiction.
\end{proof}

\subsection{Vandermondes and Tanisaki ideals}
The authors are unaware of a representation theoretic description 
of $V_n(\aaa)$ when $0 < s < r$ and $\aaa = (s-1, \dots , s-1)$ is a constant 
sequence of length $r$.
However, we can describe $V_n(\aaa)$ when $\aaa = (0, \dots, 0)$ is the length $r$ 
zero sequence.
In order to state this result, we will need a couple more definitions.

For any subset $S \subseteq [n]$ and any $d \geq 0$, let $e_d(S) \in \QQ[\xx_n]$ be the degree $d$ elementary
symmetric polynomial in the variable set $\{ x_i \,:\, i \in S \}$.  
We adopt the convention $e_d(S) = 0$ whenever $d > |S|$.

Let $\lambda \vdash n$.  The {\em Tanisaki ideal} $I_{\lambda} \subseteq \QQ[\xx_n]$ is defined
as follows.
Write the {\em conjugate} partition to $\lambda$ as
$\lambda' = (\lambda'_1 \geq \lambda'_2 \geq \cdots \geq \lambda'_n)$,
where we will have trailing zeros unless $\lambda = (n)$.
The ideal $I_{\lambda}$ has generating set
\begin{equation}
\bigcup_{i = 1}^n \{ e_d(S) \,:\, |S| = i \text{ and } d >  i - (\lambda'_{n-i+1} + \lambda'_{n-i+2} + \cdots + \lambda'_n)  \}.
\end{equation}
This ideal was used by Tanisaki \cite{Tanisaki} to present the cohomology of the 
{\em Springer fiber} $\mathcal{B}_{\lambda}$.
We let $R_{\lambda} := \QQ[\xx_n]/I_{\lambda}$ be the corresponding quotient ring, which is 
a graded $\symm_n$-module.

\begin{proposition}
\label{zero-tanisaki}
Let $r < n$ be positive integers and let $\aaa = (0, \dots, 0)$ be the length $r$ zero sequence.
Then 
\begin{equation}
\grFrob(V_n(\aaa); q) = (\rev_q \circ \omega) \grFrob(R_{\lambda}; q)
\end{equation}
where $\lambda$ is the hook-shaped partition $(r + 1, 1, 1, \dots, 1) \vdash n$.
\end{proposition}

\begin{proof}  
Write $k = n - r$.
We begin by showing that $I_{\lambda} \subseteq \ann_{\QQ[\xx_n]} \Delta_n(\aaa)$.

The ideal $I_{\lambda}$ is generated by:
\begin{enumerate}
\item the elementary symmetric polynomials $e_1, e_2, \dots, e_n$ in the full set of variables 
$\{x_1, \dots, x_n\}$ and
\item products of $k$ distinct variables $x_{i_1} \cdots x_{i_k}$.
\end{enumerate}
We show that each of these generators annihilates $\Delta_n(\aaa)$.
To do this, we adopt the notation in the proof of Lemma~\ref{sign-reversing-lemma}.

The generators of type (1) annihilate $\Delta_n(\aaa)$ by an argument similar to that 
in the proof of Lemma~\ref{sign-reversing-lemma}.
The key observation is that, since $\aaa$ is the zero sequence, all of the $\bullet$'s in any 
$d$-dotted $\aaa$-staircase must be in columns which do not contain a $\theta$.  
The generators of type (2) annihilate $\Delta_n(\aaa)$ because in any monomial appearing in
$\Delta_n(\aaa)$ there are only $k-1$ $x$-variables with positive exponents.
This completes the proof that $I_{\lambda} \subseteq \ann_{\QQ[\xx_n]} \Delta_n(\aaa)$.

We must show that 
\begin{equation}
\dim V_n(\aaa) = \dim(\QQ[\xx_n]/\ann_{\QQ[\xx_n]} \Delta_n(\aaa)) \geq 
\dim(\QQ[\xx_n]/I_{\lambda}) = \dim R_{\lambda}.
\end{equation}
The quantity $\dim R_{\lambda}$ has the following combinatorial description.
An  {\em $(n,k)$-hook staircase} is a shuffle of $(k-1, k-2, \dots, 1, 0)$ and the length 
$r$ zero sequence $(0, 0, \dots, 0)$.
For example, the $(5,3)$-hook staircases are the shuffles of $(2,1,0)$ and $(0,0)$:
\begin{equation*}
(2,1,0,0,0), \, 
(2,0,1,0,0), \, 
(2,0,0,1,0), \,
(0,2,1,0,0), \,
(0,2,0,1,0), \, \text{and }
(0,0,2,1,0).
\end{equation*}
A sequence $(c_1, \dots, c_n)$ is {\em $(n,k)$-hook-substaircase} if it is componentwise $\leq$
some $n,k$-hook staircase.  

It is known \cite{GP} that 
the set 
\begin{equation}
\{ x_1^{c_1} \cdots x_n^{c_n} \,:\, (c_1, \dots, c_n) \text{ is $(n,k)$-hook-substaircase} \}
\end{equation}
descends to a basis for $R_{\lambda}$.
As in Theorems~\ref{delta-vandermonde-theorem} and \ref{higher-constant-sequences}, we 
start with a family of linearly independent superpolynomials.

\noindent
{\bf Observation:}
{\it The subset
\begin{equation}
\left\{
\left( \sum_{j = 1}^{r} (-1)^{j-1} \theta_{i_1} \cdots \widehat{\theta_{i_j}} \cdots \theta_{i_r} \theta_n \right) +
(-1)^r \theta_{i_1} \cdots \theta_{i_r} \,:\, 1 \leq i_1 < \cdots < i_r \leq n-1
\right\}
\end{equation}
is linearly independent.}

We show that the set 
\begin{equation}
\{ \partial_1^{c_1} \cdots \partial_n^{c_n} \Delta_n(\aaa) \:\, (c_1, \dots, c_n) \text{ is
$(n,k)$-hook-substaircase} \} \subseteq V_n(\aaa)
\end{equation}
is linearly independent.
Indeed, suppose we had a dependence relation
\begin{equation}
\label{faux-identity}
\sum_{(c_1, \dots, c_n)} \gamma_{(c_1, \dots, c_n)}  \partial_1^{c_1} \cdots \partial_n^{c_n} \Delta_n(\aaa)  = 0
\end{equation}
where the sum is over $(n,k)$-hook-substaircases $(c_1, \dots, c_n)$ and the numbers
$\gamma_{(c_1, \dots, c_n)}$ are not all zero.  As before, let
$(d_1, \dots, d_n)$ be the unique $(n,k)$-hook-substaircase such that 
\begin{enumerate}
\item we have $\gamma_{(d_1, \dots, d_n)} \neq 0$,
\item subject to (1) the number $d_1 + \cdots + d_n$ is minimal, and 
\item subject to (1) and (2) the sequence $(d_1, \dots, d_n)$ is lexicographically least.
\end{enumerate}
Let $(d'_1, \dots, d'_n)$ be any $(n,k)$-hook staircase with $d_i \leq d'_i$ for all $i$ and set $p_i := d'_i - d_i$.
An argument similar to that of Lemma~\ref{lower-bound-lemma} yields the contradiction
$\gamma_{(d_1, \dots, d_n)} = 0$ upon application of $\partial_1^{p_1} \cdots \partial_n^{p_n}$ to both sides
of Equation~\eqref{faux-identity}; we leave the details to the reader.
\end{proof}

\subsection{A positroid superspace quotient}
In recent work related to an operation on symmetric functions and vector bundles
called `Chern plethysm', Billey, Rhoades, and Tewari \cite{BRT} defined 
a quotient of superspace which gives a bigraded refinement of an action of $\symm_n$
on size $n$ positroids.
In this subsection we use superpolynomials similar to 
$\aaa$-superspace Vandermondes to give an alternative model for their module.

Following \cite{BRT}, we let $J_n \subseteq \QQ[\xx_n, {\bm \theta_n}]$ 
be the bihomogeneous ideal
\begin{equation}
J_n := \langle x_1 \theta_1, x_2 \theta_2, \dots, x_n \theta_n, e_1, e_2, \dots, e_n \rangle,
\end{equation}
where the elementary symmetric polynomials $e_d = e_d(\xx_n)$ are in the $x$-variables.
Let $S_n := \QQ[\xx_n, {\bf \theta_n}]/J_n$ be the corresponding superspace quotient.
The ring $S_n$ is a bigraded $\symm_n$-module.
By \cite[Thm. 5.3]{BRT}, we have
\begin{equation}
\label{positroid-graded-frobenius}
\grFrob(S_n; q, z) = \sum_{r = 0}^n z^r \cdot e_r(\xx) \cdot \rev_q Q'_{(1^{n-r})}(\xx; q),
\end{equation}
where $q$ tracks $x$-degree and $z$ tracks $\theta$-degree.
Recall that $\rev_q Q'_{(1^{n-r})}(\xx; q)$ is the graded Frobenius image
of the coinvariant ring $R_{n-r}$ attached to $\symm_{n-r}$.

The module $S_n$ is related to positroids.
A {\em positroid} of size $n$ is a length $n$ sequence $p_1 \dots p_n$ of nonnegative integers 
which contains $r$ copies of $0$ (for some $0 \leq r \leq n$) and a single copy of $1, 2, \dots, n-r$.
Let $P_n$ be the family of positroids of size $n$. 
\footnote{A size $n$ positroid is more typically defined as a permutation in $\symm_n$ whose fixed points
are colored either black or white, but these objects are in bijection with $P_n$.}
For example, 
\begin{equation*}
P_3 = \{ 123, 213, 132, 231, 312, 321, 120, 210, 102, 201, 012, 021, 001, 010, 100, 000 \}.
\end{equation*}
By \cite[Prop. 5.2, Thm. 5.3]{BRT} the dimension of $S_n$ counts size $n$ positroids:
\begin{equation}
\dim(S_n) = |P_n| = \sum_{r = 0}^n \frac{n!}{r!}.
\end{equation}
The $\QQ$-vector space $\QQ[P_n]$ with basis $P_n$ carries an action of $\symm_n$ defined on adjacent
transpositions by
\begin{equation}
(i,i+1). p_1 \dots p_i p_{i+1} \dots  p_n := \pm ( p_1 \dots p_{i+1} p_i \dots p_n), \quad 1 \leq i \leq n-1
\end{equation}
where the sign is $-$ if $p_i = p_{i+1}$ and $+$ if $p_i \neq p_{i+1}$.
By \cite[Prop. 5.2, Thm. 5.3]{BRT} we have
\begin{equation}
\Frob(S_n) = \Frob(\QQ[P_n]) 
\end{equation}
so that $S_n$ gives a bigraded refinement of this $\symm_n$-action on positroids.

We will give an alternative model for $S_n$ as a bigraded subspace of $\QQ[\xx_n, {\bm \theta}_n]$
rather than as a quotient ring.
For $0 \leq k \leq n $ we define $\rho_{n,k} \in \QQ[\xx_n, {\bf \theta}_n]$ to be the 
superpolynomial
\begin{equation}
\rho_{n,k} := \varepsilon_k \cdot (x_1^{k-1} \cdots x_{k-1}^1 x_k^0) \cdot \theta_{k+1} \cdots \theta_{n-1} \theta_n.
\end{equation}
Here $\varepsilon_k = \sum_{w \in \symm_k} \sign(w) \cdot w$ acts on the subscripts of the first $k$ 
variables.
In particular, the element $\rho_{n,n} = \Delta_n$ is the classical Vandermonde.

\begin{defn}
\label{w-space-definition}
Let $M_n$ be the smallest linear subspace of $\QQ[\xx_n, {\bm \theta}_n]$ which
\begin{itemize}
\item contains each of the superpolynomials $\rho_{n,0}, \rho_{n,1}, \dots, \rho_{n,n}$,
\item is closed under the action of $\symm_n$, and
\item is closed under the partial derivatives $\partial_1, \partial_2, \dots, \partial_n$
acting on the $x$-variables.
\end{itemize}
\end{defn}

The vector space $M_n$ is a bigraded $\symm_n$-module.

\begin{proposition}
\label{w-structure}
The bigraded $\symm_n$-modules $S_n$ and $M_n$ are isomorphic.  Equivalently, we have
\begin{equation}
\label{w-bigraded-frobenius}
\grFrob(M_n; q, z) = \sum_{r = 0}^n z^r \cdot e_r(\xx) \cdot \rev_q Q'_{(1^{n-r})}(\xx; q).
\end{equation}
\end{proposition}

\begin{proof}
For $0 \leq r \leq n$ let $M_{n-r}$ be the smallest subspace of $M_n$ containing $\rho_{n,n-r}$
which is closed under the action of $\symm_n$ and the partial derivatives $\partial_1, \dots, \partial_n$.
Then $M_{n-r}$ is the $\theta$-homogeneous piece of $M_n$ of degree $r$, so that 
\begin{equation}
\grFrob(M_n; q, z) = \sum_{r = 0}^n z^r \cdot \grFrob(M_{n-r}; q),
\end{equation}
so it suffices to verify
\begin{equation}
\label{w-goal-equation}
\grFrob(M_{n-r}; q) = e_r(\xx) \cdot  \rev_q Q'_{(1^{n-r})}(\xx; q),
\end{equation}
where $q$ tracks $x$-degree.

Equation~\eqref{w-goal-equation} states that $M_{n-r}$ is the induction product of the sign
representation of $\symm_r$ with the coinvariant algebra attached to $\symm_{n-r}$.
Indeed, for any $I = \{i_1 < \cdots < i_r \} \subseteq [n]$ with $|I| = r$ and complement
$J = [n] - I = \{j_1 < \cdots < j_k \}$,
let $M_I \subseteq M_{n-r}$ be the 
smallest linear subspace such that 
\begin{itemize}
\item we have
$\left[ \left( \sum_{w \in \symm_J} \sign(w) \cdot w \right) \cdot x_{j_1}^{k-1} \cdots x_{j_{k-1}}^1 x_{j_k}^0 \right] \cdot
\theta_{i_1} \cdots \theta_{i_r} \in M_I$, and
\item $M_I$ is closed under the partial derivatives $\partial_1, \dots, \partial_n$.
\end{itemize} 
We have the vector space direct sum decomposition
\begin{equation}
\label{w-direct-sum}
M_{n-r} = \bigoplus_{\substack{I \subseteq [n] \\ |I| = r}} M_I.
\end{equation}
Taking $I = [r] = \{1, 2, \dots, r \}$, the space $M_{[r]}$ is a graded $\symm_r \times \symm_{n-r}$-module
isomorphic to $\sign_r \otimes R_{n-r}$, where $\sign_r$ is the 1-dimensional sign representation of $\symm_r$
and $R_{n-r}$ is the coinvariant ring attached to $\symm_{n-r}$.  
Equation~\eqref{w-direct-sum} leads to the identification of $M_{n-r}$ as the induction product
\begin{equation}
\label{w-induction-product}
M_{n-r} \cong \sign_r \circ R_{n-r},
\end{equation}
which implies Equation~\eqref{w-goal-equation}.
\end{proof}

\section{Antisymmetric differentiation}
\label{Antisymmetric differentiation}

The singly graded modules $V_n(\aaa)$  are based on an action of the partial 
derivative operators $\partial_i$ acting on the $x$-variables in $\QQ[\xx_n, {\bm \theta_n}]$.
The goal of this section is to describe how new operators $\partial^{\theta}_i$ acting on the $\theta$-variables
can be used to build new doubly graded modules $W_n(\aaa)$.
The modules $W_n(\aaa)$ will contain the $V_n(\aaa)$ as their top antisymmetric  
components and will exhibit a new kind of duality which is invisible at the level of $V_n(\aaa)$.

\subsection{The operators $\partial_i^{\theta}$}
How can we differentiate with respect to a skew-commuting variable?
Recall from Section~\ref{Introduction} that 
 $\partial^{\theta}_i: \QQ[\xx_n, {\bm \theta}_n] \rightarrow \QQ[\xx_n, {\bm \theta_n}]$ 
is the $\QQ[\xx_n]$-linear operator determined on $\theta$-monomials by the rule
\begin{equation}
\partial^{\theta}_i: \theta_{j_1} \cdots \theta_{j_r} \mapsto
\begin{cases}
(-1)^{k-1} \theta_{j_1} \cdots \widehat{\theta_{j_k}} \cdots \theta_{j_r} &
\text{if $j_k = i$,} \\
0 & \text{if $i \notin \{j_1 < \dots < j_k \}$}
\end{cases}
\end{equation}
for all $1 \leq j_1 < \cdots < j_r \leq n$.
The sign $(-1)^{k-1}$ ensures that $\partial^{\theta}_i$ is a well-defined $\QQ[\xx_n]$-endomorphism of 
$\QQ[\xx_n, {\bm \theta_n}]$.  
For another characterization of these operators, see Remark~\ref{uniqueness-remark}.
We begin with some basic identities satisfied by the $\partial^{\theta}_i$.

\begin{lemma}
\label{basic-partial-lemma}
Let $1 \leq i, j \leq n$.  We have the following identities of  operators on 
$\QQ[\xx_n, {\bm \theta_n}]$.
\begin{equation}
\label{derivative-superspace-relations}
\partial_i \partial_j = \partial_j \partial_i, \quad
\partial_i \partial^{\theta}_j = \partial^{\theta}_j \partial_i, \quad
\partial^{\theta}_i \partial^{\theta}_j = - \partial^{\theta}_j \partial^{\theta}_i.
\end{equation}
Furthermore, if $w \in \symm_n$ we have the operator identities 
\begin{equation}
\label{permutation-derivative-relations}
w \cdot \partial_i \cdot w^{-1} = \partial_{w(i)}, \quad
w \cdot \partial^{\theta}_i \cdot w^{-1} = \partial^{\theta}_{w(i)}.
\end{equation}
Finally, if $f, g \in \QQ[\xx_n, {\bm \theta_n}]$ and $f$ is concentrated in $\theta$-degree $r$ we have
\begin{equation}
\label{leibniz}
\partial^{\theta}_i(f g) = \partial^{\theta}_i(f) g + (-1)^r f \partial^{\theta}_i(g).
\end{equation}
\end{lemma}

\begin{proof}
We begin with \eqref{derivative-superspace-relations}.
The first assertion is the standard commutativity of mixed partials.
The second follows because $\partial_i$ acts on $x$-variables and $\partial_i^{\theta}$
acts on $\theta$-variables.
The third can be verified directly on any $\theta$-monomial $\theta_{k_1} \cdots \theta_{k_r}$
for $1 \leq k_1 < \cdots < k_r \leq n$. There are two cases depending on whether
$i, j \in \{k_1, \dots, k_r\}$; 
we leave the 
details to the reader.

We turn our attention to \eqref{permutation-derivative-relations}. The first assertion of 
\eqref{permutation-derivative-relations} has already been observed. For the second, it suffices to consider the 
case  $w = (p, p+1)$ is an adjacent transposition in $\symm_n$ for some $1 \leq p \leq n-1$.
Given $1 \leq k_1 < \cdots < k_r \leq n$, it is enough to show that
\begin{equation}
\label{permutation-lemma-goal}
w \cdot \partial^{\theta}_i \cdot w^{-1} (\theta_{k_1} \cdots \theta_{k_r}) = 
\partial^{\theta}_{w(i)} (\theta_{k_1} \cdots \theta_{k_r}), \quad \text{where $w = (p, p+1)$.}
\end{equation}
We have 
\begin{equation}
\label{condition}
w(i) \in \{ k_1 , \dots , k_r \} \text{ if and only if }
i \in \{ w^{-1}(k_1), \dots w^{-1}(k_r) \}.
\end{equation}
If \eqref{condition} does not hold, then both sides of \eqref{permutation-lemma-goal} equal 0, so assume
\eqref{condition} does hold.
If $i \notin \{p, p+1\}$ then $w(i) = i$ and both sides of \eqref{permutation-lemma-goal} equal
$(-1)^{s-1} \theta_{k_1} \cdots \widehat{\theta_{k_s}} \cdots \theta_{k_r}$ 
where $i = k_s$.
If $i = p$ then $w(i) = i+1$; we leave it for the reader 
to check that both sides of \eqref{permutation-lemma-goal} equal
$(-1)^{s-1} \theta_{k_1} \cdots \widehat{\theta_{k_s}} \cdots \theta_{k_r}$ 
where $i+1 = k_s$ (there are two cases depending on whether $i \in \{k_1, \dots, k_r\}$).
The case $i = p+1$ is similar to the case $i = p$ and left to the reader.

Since $\partial_i^{\theta}$ is a map of $\QQ[\xx_n]$-modules,
 \eqref{leibniz} can be verified in the case where $f, g$ are monomials
in the $\theta$-variables.  We leave the details to the reader.
\end{proof}

\begin{remark}
\label{uniqueness-remark}
We will not use the Leibniz relation \eqref{leibniz} in this paper, but 
the operator $\partial_i^{\theta}$ can be characterized as the unique $\QQ$-linear endomorphism of 
$\QQ[\xx_n, {\bm \theta_n}]$ which satisfies 
\begin{equation}
\partial^{\theta}_i(x_j) = 0 \quad \text{and} \quad
\partial^{\theta}_i(\theta_j) = \delta_{i,j}
\end{equation}
(where $1 \leq j \leq n$ and $\delta_{i,j}$ is the Kronecker delta)
together with \eqref{leibniz}.
\end{remark}

By \eqref{derivative-superspace-relations}, superspace $\QQ[\xx_n, {\bm \theta_n}]$ acts on itself by the rule
\begin{equation}
f \cdot g := \partial(f)(g) \quad \text{for $f, g \in \QQ[\xx_n, {\bm \theta_n}]$},
\end{equation}
where $\partial(f)$ is obtained from $f$ by replacing every $x_i$ by $\partial_i$ and every $\theta_i$ by 
$\partial^{\theta}_i$.
This extends the action of $\QQ[\xx_n]$ on superspace discussed earlier.
We repeat Definition~\ref{double-cone-definition}, which introduces the bigraded modules of study.

\vspace{0.1in}
\noindent
{\bf Definition 1.4.}
{\em Suppose $n = k + r$ for $k, r \geq 0$ 
 and let $\aaa \in (\ZZ_{\geq 0})^r$.  Let $W_n(\aaa)$
to be the smallest $\QQ$-linear subspace of $\QQ[\xx_n, {\bm \theta_n}]$
containing $\Delta_n(\aaa)$ which is closed 
under the $n$ partial derivative operators
$\partial_1, \dots, \partial_n$ as well as the $n$ operators
$\partial^{\theta}_1, \dots, \partial^{\theta}_n$.}
\vspace{0.1in}

Definition~\ref{double-cone-definition} can also be interpreted as saying that 
$W_n(\aaa)$ is the cyclic $\QQ[\xx_n, {\bm \theta_n}]$-submodule
of $\QQ[\xx_n, {\bm \theta_n}]$ generated by $\Delta_n(\aaa)$.
Since $\Delta_n(\aaa)$ is alternating, the relations \eqref{permutation-derivative-relations}
imply that $W_n(\aaa)$ is closed under the action of $\symm_n$ and therefore a bigraded $\symm_n$-module.
We have $V_n(\aaa) \subseteq W_n(\aaa)$; in fact, $V_n(\aaa)$ is the $\theta$-homogeneous piece of 
$W_n(\aaa)$ of $\theta$-degree $r$.

The spaces $W_n(\aaa)$ have nicer algebraic properties than the $V_n(\aaa)$.
For example, we will see that $W_n(\aaa)$ may be presented as a bigraded quotient of superspace.
We repeat the relevant
Definition~\ref{ideal-ring-definition}.

\vspace{0.1in}
\noindent
{\bf Definition 1.5.}
{\em Suppose $n = k + r$ for $k, r \geq 0$ and let $\aaa = (a_1, \dots, a_r) \in (\ZZ_{\geq 0})^r$.  
Let $I_n(\aaa) \subseteq \QQ[\xx_n, {\bm \theta_n}]$ be the ideal
$I_n(\aaa) := \{ f \in \QQ[\xx_n, {\bm \theta_n}] \,:\, f \cdot \Delta_n(\aaa) = 0 \}$
and let
$R_n(\aaa) := \QQ[\xx_n, {\bm \theta_n}]/I_n(\aaa)$
be the corresponding quotient ring.}
\vspace{0.1in}

The bigraded vector space $W_n(\aaa)$ and the bigraded ring $R_n(\aaa)$ posses 
a duality which is invisible at the level of $V_n(\aaa)$.
Establishing this duality -- as well as the equivalence of $W_n(\aaa)$ and $R_n(\aaa)$
as doubly graded $\symm_n$-modules -- is our next goal.

%We may consider the annihilator 
%\begin{equation}
%\ann_{\QQ[\xx_n, {\bm \theta_n}]} \Delta_n(\aaa)=
%\{ f \in \QQ[\xx_n, {\bm \theta_n}] \,:\, f \cdot \Delta_n(\aaa) = 0 \}
%\end{equation}
%of $\Delta_n(\aaa)$ inside $\QQ[\xx_n, {\bm \theta_n}]$.
%This is a (two-sided) ideal in $\QQ[\xx_n, {\bm \theta_n}]$ which is closed under the action
%of $\symm_n$.
%Proposition~\ref{annihilator-twist} has the following analogue whose proof is similar and omitted.

%\begin{proposition}
%\label{super-annihilator-twist}
%For any $0 \leq r \leq n$ and any sequence $\aaa \in (\ZZ_{\geq 0})^r$ we have 
%\begin{equation}
%\grFrob(\QQ[\xx_n, {\bm \theta_n}]/\ann_{\QQ[\xx_n, {\bm \theta_n}]} \Delta_n(\aaa); q, z) =
%(\rev_q \circ \rev_z \circ \omega)
%\grFrob(W_n(\aaa); q, z)
%\end{equation}
%where $q$ tracks $x$-degree and $z$ tracks $\theta$-degree.
%\end{proposition}

%In Proposition~\ref{super-annihilator-twist} the operator $\rev_q$ acts on formal power series in 
%$\QQ[[q, z, x_1, x_2, \dots ]]$ with finite $q$-degree by regarding them
%as polynomials in $\QQ[[z, x_1, x_2, \dots ]][q]$.  A similar remark applies to $\rev_z$.

\subsection{A duality of $W_n(\aaa)$}
To better understand the duality 
enjoyed by the $W_n(\aaa)$, let us look at some examples of their bigraded 
Frobenius images.
The symmetric function $\grFrob(W_3(1); q, z)$
is displayed in a matrix below, where the entry in row $i$ and column $j$ gives the coefficient
of $z^i q^j$ in the Schur basis.

\begin{equation*} 
\begin{pmatrix}
s_3 & s_3 + s_{21} & s_{21} \\
s_{21} & s_{21} + s_{111} & s_{111}
\end{pmatrix}
\end{equation*}
The symmetric function $\grFrob(W_4(1,1);q,z)$ is similarly displayed below.

\begin{equation*}
\begin{pmatrix}
s_4 & s_4 + s_{31} & s_4 + s_{31} + s_{22} & s_{31} \\
s_{31} & 2 s_{31} + s_{22} + s_{211} & s_{31} + s_{22} + 2 s_{211} & s_{211} \\
s_{211} & s_{22} + s_{211} + s_{1111} & s_{211} + s_{1111} & s_{1111}
\end{pmatrix}
\end{equation*}
Finally, we display the symmetric function $\grFrob(W_4(2,1);q,z)$.
\begin{equation*}
\begin{footnotesize}
\begin{pmatrix}
s_4 & s_4 + s_{31} & s_4 + 2 s_{31} + s_{22} & s_4 + 2 s_{31} + s_{22} + s_{211} & s_{31} + s_{211} \\
s_4 + s_{31} & s_4 + 3 s_{31} + s_{22} + s_{211} & 3 s_{31} + 3 s_{22} + 3 s_{211} & 
s_{31} + s_{22} + 3 s_{211} + s_{1111} & s_{211} + s_{1111} \\
s_{31} + s_{211}  & s_{31} + s_{22} + 2 s_{211} + s_{1111} & s_{22} + 2 s_{211} + s_{1111} & s_{211} + s_{1111} &
s_{1111}
\end{pmatrix}
\end{footnotesize}
\end{equation*}

These tables have the property 
that if they are rotated $180^{\circ}$, the effect is the same as if the $\omega$
involution were applied to each entry.
The goal of this section is to prove (Corollary~\ref{duality-corollary}) that this is a general
phenomenon.

The action $f \cdot g := \partial(f)(g)$ of superspace on itself yields
a bilinear form on $\QQ[\xx_n, {\bm \theta_n}]$.  More precisely, given
$f, g \in \QQ[\xx_n, {\bm \theta_n}]$ we define a rational number $\langle f, g \rangle \in \QQ$ by 
\begin{equation}
\langle f, g \rangle := \text{constant term of $f \cdot g$} = \text{constant term of $\partial(f)(g)$} .
\end{equation}
In particular, if $f$ and $g$ are homogeneous superpolynomials we have
$\langle f, g \rangle = 0$ unless $f$ and $g$ have the same $x$-degree and the same $\theta$-degree.
%It should be clear that 
%$\langle - , - \rangle: \QQ[\xx_n, {\bm \theta_n}] \times \QQ[\xx_n, {\bm \theta_n}] \rightarrow \QQ$
%is bilinear.  

As an example of the form $\langle - , - \rangle$ we calculate 
$\langle x_1^3 x_2^2 \theta_1 \theta_2, x_1^3 x_2^2 \theta_1 \theta_2 \rangle = - 3! 2!$,
where the minus sign comes from the  
action of the operator $\partial^{\theta}_2$.
In particular, the form $\langle - , - \rangle$ is not positive definite.
However, the form $\langle - , - \rangle$ enjoys a graded version of positive definiteness.

\begin{lemma}
\label{bilinear-lemma}
The bilinear form $\langle - , - \rangle$ defined above is symmetric. Furthermore we have:
\begin{enumerate}
\item  Let $f \in \QQ[\xx_n, {\bm \theta_n}]$ be a $\theta$-homogeneous superpolynomial
of $\theta$-degree $r$. Then
\begin{equation*}
\begin{cases}
\langle f, f \rangle \geq 0 & \text{if $r \equiv 0, 1$ (mod $4$),} \\
\langle f, f \rangle \leq 0 & \text{if $r \equiv 2, 3$ (mod $4$),}
\end{cases}
\end{equation*}
with equality if and only if $f = 0$.
\item For any $\theta$-homogeneous $f, g, f', g' \in \QQ[\xx_n, {\bm \theta_n}]$ we have
\begin{equation*}
\langle f, (f' \cdot g') \cdot g \rangle = \pm \langle f \cdot g, f' \cdot g' \rangle.
\end{equation*}
where the sign is determined by the $\theta$-degrees of $f, g, f', g'$.
\end{enumerate}
\end{lemma}

The proof of Lemma~\ref{bilinear-lemma} (2) is notationally cumbersome, but worth it.
This fact will be crucial for our duality result (Theorem~\ref{duality-theorem}) below.

\begin{proof}
Let $m = x_1^{a_1} \cdots x_n^{a_n} \theta_{i_1} \cdots \theta_{i_r}$ and
$m' = x_1^{a'_1} \cdots x_n^{a'_n} \theta_{i'_1} \cdots \theta_{i'_s}$ be two supermonomials
(so that $i_1 < \cdots < i_r$ and $i'_1 < \cdots < i'_s$).  A direct computation gives
\begin{equation}
\label{monomial-inner-product}
\langle m, m' \rangle = \begin{cases}
(-1)^{{r \choose 2}} a_1! \cdots a_n! & \text{if $m = m'$,} \\
0 & \text{otherwise.}
\end{cases}
\end{equation}
This shows that $\langle - , - \rangle$ is symmetric.
The supermonomials are orthogonal with respect to $\langle - , - \rangle$.
Since ${r \choose 2}$ is even when $r \equiv 0, 1$ (mod $4$) and
 odd when $r \equiv 2, 3$ (mod $4$), assertion (1) of the lemma is true.
 
 Assertion (2) is linear in each of the four arguments $f, g, f', g'$ separately, so it suffices 
 to consider the case where $f, g, f', g'$ are supermonomials.  To this end, write
 \begin{center}
 $\begin{array}{cc}
 f = x_1^{a_1} \cdots x_n^{a_n} \theta_{i_1} \cdots \theta_{i_r} &
 f' = x_1^{a'_1} \cdots x_n^{a'_n} \theta_{i'_1} \cdots \theta_{i'_{r'}} \\
  g = x_1^{b_1} \cdots x_n^{b_n} \theta_{j_1} \cdots \theta_{j_s} &
 g' = x_1^{b'_1} \cdots x_n^{b'_n} \theta_{j'_1} \cdots \theta_{j'_{s'}}
 \end{array}$
 \end{center}
 where the sequences $i, i', j, j'$ of $\theta$-subscripts are increasing.
 We introduce the sets of $\theta$-subscripts 
 $I = \{i_1 , \dots , i_r \}, I' = \{i'_1 , \dots , i'_{r'} \}, J = \{j_1 , \dots, j_s \},$
 and $J' = \{j'_1 , \dots , j'_{s'} \}$ and let 
 $a = (a_1, \dots, a_n), a' = (a'_1, \dots, a'_n), b = (b_1, \dots, b_n), 
 b' = (b'_1, \dots, b'_n)$ be the sequences of $x$-exponents.
 If $c = (c_1, \dots, c_n)$ is a vector of nonnegative integers we adopt the
 shorthands $c! := c_1! \cdots c_n!$ and
 $x^c := x_1^{c_1} \cdots x_n^{c_n}$.
 If $K = \{ k_1 < \cdots < k_q\} \subseteq [n]$ we write
 $\theta_K := \theta_{k_1} \cdots \theta_{k_q}$.
 We begin by verifying $|\langle f, (f' \cdot g') \cdot g \rangle| = |\langle f \cdot g, f' \cdot g' \rangle|$
 and check the sign assertion afterward.
 
 We have $f \cdot g = 0$ unless 
 \begin{quote}
 $(\spadesuit)$
 $I \subseteq J$ and
we have the componentwise inequality $a \leq b$.  
\end{quote}
If $(\spadesuit)$ holds we have
\begin{equation*}
f \cdot g = \pm \frac{b!}{(b - a)!}
x^{b-a} \theta_{J - I}
\end{equation*}
where the sign depends on the sets $I$ and $J$. 
Analogous remarks apply to $f' \cdot g'$.
In summary, we see that $\langle f \cdot g, f' \cdot g' \rangle = 0$
unless 
\begin{quote}
$(\heartsuit)$ the condition $(\spadesuit)$ holds and in addition 
we have the set containment $I' \subseteq J'$, the vector inequality $a' \leq b'$,
the set equality $J - I = J' - I'$, and the vector equality
$b - a = b' - a'$.
\end{quote}
If  $(\heartsuit)$ holds, we have
\begin{equation}
\label{inner-product-one}
\langle f \cdot g, f' \cdot g' \rangle = \pm
\frac{b! \cdot b'!}{(b-a)!}.
\end{equation}
%where (by assertion (1) of the lemma) the sign is determined by the common size $|I - J|$ of the
%set $I - J = I' - J'$, and therefore by the $\theta$-degrees of $f, g, f', g'$.

On the other hand, we have $(f' \cdot g') \cdot g = 0$ unless 
\begin{quote}
$(\diamondsuit)$
we have the set containments $I' \subseteq J'$ and
$(J' - I') \subseteq J$ as well as  the vector inequalities $a' \leq b'$
and $(b' - a') \leq b$.
\end{quote}
If $(\diamondsuit)$ holds we have
\begin{align}
(f' \cdot g') \cdot g  &= 
\left( \pm \frac{b'!}{(b' - a')!} x^{b' - a'} \theta_{J' - I'} \right) \cdot g \\
&= \pm \frac{b! \cdot b'!}{(b' - a')! (b - b' + a')!} x^{b - b' + a'} \theta_{J - (J' - I')}.
\end{align}
It follows that $\langle f, (f' \cdot g') \cdot g \rangle = 0$ unless
\begin{quote}
($\clubsuit$)  the condition $(\diamondsuit)$ holds and in addition we have the 
set equality $I = J - (J' - I')$ and the vector equality $a = b - b' + a'$.
\end{quote}
If $(\clubsuit)$ holds then
\begin{equation}
\label{inner-product-two}
\langle f, (f' \cdot g') \cdot g \rangle = 
\pm \frac{b! \cdot b'!}{(b' - a')! (b - b' + a')!} (b - b' + a')! 
= \pm \frac{b! \cdot b'!}{(b-a)!},
\end{equation}
where the second equality used $b' - a' = b - a$.
%and the sign
%(by assertion (1) of the lemma) is determined by the $\theta$-degrees of $f, g, f', g'$.

The equality $|\langle f, (f' \cdot g') \cdot g \rangle| = |\langle f \cdot g, f' \cdot g' \rangle|$
follows by comparing Equation~\eqref{inner-product-one}
with Equation~\eqref{inner-product-two} and checking that conditions $(\heartsuit)$
and $(\clubsuit)$ are equivalent.

Our final task is to verify the sign claim in assertion (2). It suffices to consider the case
where $f, f', g', g'$ are monomials in the
$\theta$-variables which satisfy the condition $(\heartsuit)$ (or the condition $(\clubsuit)$).
That is, we have 
\begin{center}
$f = \theta_I, g = \theta_J, f' = \theta_{I'},$ and $g' = \theta_{J'}$ 
with $I \subseteq J$, $I' \subseteq J'$, and $J - I = J' - I'$.
\end{center}

We  introduce  notation to carefully keep track of signs. Recall that $J = \{j_1 < \cdots < j_s\}$.
Since $I \subseteq J$, there are indices $1 \leq p_1 < \cdots < p_r \leq s$ such that
\begin{equation*}
I = \{ i_1 < \cdots < i_r \} = \{ j_{p_1} < \cdots < j_{p_r} \} \subseteq J = \{j_1 < \cdots <  j_s \}.
\end{equation*}
Let $1 \leq q_1 < \cdots < q_{s-r} \leq s$ be the complement of these indices, i.e.
\begin{equation*}
\{1 < 2 < \dots < s \} - \{p_1 < \cdots < p_r \} = \{ q_1 < \cdots  < q_{s-r} \}.
\end{equation*}
Similarly, define indices $1 \leq p'_1 < \cdots < p'_{r'} \leq s'$ so that 
\begin{equation*}
I' = \{ i'_1 < \cdots < i'_{r'} \} = \{ j'_{p'_1} < \cdots < j_{p'_{r'}} \} \subseteq J = \{j'_1 < \cdots <  j'_{s'} \}
\end{equation*}
and define $1 \leq q'_1 < \cdots < q'_{s'-r'} \leq s'$  by
\begin{equation*}
\{1 < 2 < \dots < s' \} - \{p'_1 < \cdots < p'_{r'} \} = \{ q'_1 < \cdots  < q'_{s'-r'} \}.
\end{equation*}

The left-hand-side of assertion (2) reads
\begin{align}
\langle \theta_I, (\theta_{I'} \cdot \theta_{J'}) \cdot \theta_J \rangle &=
(-1)^{p'_1 + \cdots + p'_{r'} - r'} \langle \theta_I, (\theta_{J'-I'}) \cdot \theta_J \rangle \\
&=
(-1)^{p'_1 + \cdots + p'_{r'} - r'} \langle \theta_I, (\theta_{J-I}) \cdot \theta_J \rangle \\
&=
(-1)^{p'_1 + \cdots + p'_{r'} - r' + q_1 + \cdots + q_{s-r} - s + r} \langle \theta_I, \theta_I \rangle \\
&= 
(-1)^{p'_1 + \cdots + p'_{r'} - r' + q_1 + \cdots + q_{s-r} - s + r + {r \choose 2}}
\end{align}
whereas the right-hand-side of assertion (2) reads
\begin{align}
\langle \theta_I \cdot \theta_J, \theta_{I'} \cdot \theta_{J'} \rangle &=
(-1)^{p_1 + \cdots + p_r - r + p'_1 + \cdots + p'_{r'} - r'} \langle \theta_{J-I}, \theta_{J'-I'} \rangle \\
&= 
(-1)^{p_1 + \cdots + p_r - r + p'_1 + \cdots + p'_{r'} - r'} \langle \theta_{J-I}, \theta_{J-I} \rangle \\
&= 
(-1)^{p_1 + \cdots + p_r - r + p'_1 + \cdots + p'_{r'} - r' + {s-r \choose 2}} 
\end{align}
where we used $J - I = J' - I'$.  The difference between the two exponents of $(-1)$ is
\begin{equation}
\begin{split}
\left[p'_1 + \cdots + p'_{r'} - r' + q_1 + \cdots + q_{s-r} - s + r + {r \choose 2} \right] - \\
\left[p_1 + \cdots + p_r - r + p'_1 + \cdots + p'_{r'} - r' + {s-r \choose 2} \right]
\end{split}
\end{equation}
which equals
\begin{equation}
\label{sign-parity}
q_1 + \cdots + q_{s-r} - s + 2 r + {r \choose 2}   \\
-p_1 - \cdots - p_r  - {s-r \choose 2}.
\end{equation}

Assertion (2) will be proved if we can show that 
the expression \eqref{sign-parity} modulo 2 only depends on $s$ and $r$.
Since ${s \choose 2} = p_1 + \cdots + p_r + q_1 + \cdots + q_{s-r} - s$, \eqref{sign-parity} is
congruent modulo 2 to 
\begin{equation}
\label{sign-parity-goal}
{s \choose 2}  + 2 r + {r \choose 2}   \\
-2p_1 - \cdots - 2p_r  - {s-r \choose 2} \equiv {s \choose 2} + {r \choose 2} + {s - r \choose 2},
\end{equation}
which completes the proof.
\end{proof}

For $0 \leq j \leq n$,
let $\QQ[\xx_n, {\bm \theta}_n]_j$ be the (infinite-dimensional) subspace of 
$\QQ[\xx_n, {\bm \theta}_n]$ consisting of superpolynomials which are $\theta$-homogeneous
of $\theta$-degree $j$.  We have a direct sum decomposition
\begin{equation}
\QQ[\xx_n, {\bm \theta}_n] = \bigoplus_{j = 0}^n \QQ[\xx_n, {\bm \theta}_n]_j.
\end{equation}
The bilinear form $\langle - , - \rangle$ on 
$\QQ[\xx_n, {\bm \theta}_n]$ may be restricted to
$\QQ[\xx_n, {\bm \theta}_n]_j$ for any $j$.
Lemma~\ref{bilinear-lemma} (1) states that this restriction will be positive definite if 
${j \choose 2}$ is even and negative definite if ${j \choose 2}$ is odd.

\begin{lemma}
\label{perp-lemma}
Suppose $V \subseteq \QQ[\xx_n, {\bm \theta}_n]_j$ is a linear subspace (for some fixed $0 \leq j \leq n$)
which is  $x$-homogeneous (i.e. for any $f \in V$, all of the $x$-homogeneous components of $f$ are 
contained in $V$).
Define a new subspace $V^{\perp} \subseteq \QQ[\xx_n, {\bm \theta}_n]_j$ by
\begin{equation}
V^{\perp} = \{ f \in \QQ[\xx_n, {\bm \theta}_n]_j \,:\, \langle f, g \rangle = 0 \text{ for all $g \in V$} \}.
\end{equation}
We have the direct sum of vector spaces $\QQ[\xx_n, {\bm \theta}_n]_j = V \oplus V^{\perp}$.
\end{lemma}

\begin{proof}
For $j \geq 0$, let $\QQ[\xx_n, {\bm \theta}_n]_{i,j}$ be the finite-dimensional subspace 
of $\QQ[\xx_n, {\bm \theta}_n]_i$ which is $x$-homogeneous of $x$-degree $i$ and let
$V_i := V \cap \QQ[\xx_n, {\bm \theta}_n]_{i,j}$.  By the assumption on $V$ we have
$V = \bigoplus_{i \geq 0} V_i$.  If we set 
\begin{equation}
V_i^{\perp} = \{ f \in \QQ[\xx_n, {\bm \theta}_n]_{i,j} \,:\, \langle f, g \rangle = 0 \text{ for all $g \in V_i$} \},
\end{equation}
Lemma~\ref{bilinear-lemma} (1) implies $\QQ[\xx_n, {\bm \theta}_n]_{i,j} = V_i \oplus V_i^{\perp}$.
Taking a direct sum over $j \geq 0$ gives the result.
\end{proof}

Informally, Lemma~\ref{perp-lemma} says that we can take orthogonal complements in superspace
as long as the subspaces in question are concentrated in a single $\theta$-degree.
This will play a key role in proving the following result.

\begin{theorem}
\label{duality-theorem}
Let $0 \leq r \leq n$ and let $\aaa = (a_1, \dots, a_r) \in \ZZ_{\geq 0}^r$.  Define a map
\begin{equation}
\iota: W_n(\aaa) \rightarrow W_n(\aaa)
\end{equation}
by the rule $\iota(f) := f \cdot \Delta_n(\aaa) = \partial(f)(\Delta_n(a))$.

The map $\iota$ is a linear automorphism of $W_n(\aaa)$ which complements both 
$x$-degree and $\theta$-degree simultaneously.  Furthermore, the map $\iota$ satisfies
\begin{equation}
\iota(w \cdot f) = \sign(w) \iota(f) \quad \text{for all $w \in \symm_n$ and $f \in W_n(\aaa)$.}
\end{equation}
\end{theorem}

\begin{proof}
The remaining statements of the theorem will follow immediately if we can show that
$\iota$ is bijective.
Since $\iota$ is $\QQ$-linear and $W_n(\aaa)$ is finite-dimensional, it is enough to check that 
$\iota$ is surjective.  

For $0 \leq j \leq r$, let $W_n(\aaa)_j$ be the $\theta$-homogeneous piece of 
$W_n(\aaa)$ of $\theta$-degree $j$.
Similarly, let $I_n(\aaa)_j$ for the $\theta$-homogeneous piece of $I_n(\aaa)$ of $\theta$-degree $j$.
We have the direct sum decompositions
\begin{equation}
W_n(\aaa) = \bigoplus_{j = 0}^r W_n(\aaa)_j \quad \text{and} \quad
I_n(\aaa) = \bigoplus_{j = 0}^r I_n(\aaa)_j.
\end{equation}

Fix $0 \leq j \leq r$.  Let $\QQ[\xx_n, {\bm \theta_n}]_j$ be the space 
of superpolynomials which are $\theta$-homogeneous of $\theta$-degree $j$.
Consider the orthogonal complement 
\begin{equation}
W_n(\aaa)_j^{\perp} = \{ f \in \QQ[\xx_n, {\bm \theta_n}]_j \,:\, \langle f, g \rangle = 0 \text{ for all $g \in W_n(\aaa)_j$}\}.
\end{equation}

\noindent
{\bf Claim:}  
{\em We have $I_n(\aaa)_j = W_n(\aaa)^{\perp}_j$.}

To prove the Claim, we consider both containments separately. For the containment $\subseteq$,
let $f \in I_n(\aaa)_j$ and $g \in W_n(\aaa)_j$.
There exists $h \in \QQ[\xx_n, {\bm \theta}_n]$ such that $g = h \cdot \Delta_n(\aaa)$.
After discarding redundant terms if necessary, we may assume that $h$ is $\theta$-homogeneous.
We have
\begin{equation}
f \cdot g = f \cdot (h \cdot \Delta_n(\aaa)) = \partial(f h) (\Delta_n(\aaa)) = \pm 
\partial(h f) (\Delta_n(\aaa)) = \pm \partial(h) (f \cdot \Delta_n(\aaa)) = 0,
\end{equation}
where the third equality used the $\theta$-homogeneity of $f$ and $h$.
Taking the constant term gives $\langle f, g \rangle = 0$ so that $f \in W_n(\aaa)^{\perp}_j$.

Now we prove the containment $\supseteq$.  Let $f \in W_n(\aaa)^{\perp}_j$.  
We want to show that $f \cdot \Delta_n(\aaa) = 0$.
Since both $f$ and
$\Delta_n(\aaa)$ are $\theta$-homogeneous, Lemma~\ref{bilinear-lemma} (1) implies that 
\begin{equation}
f \cdot \Delta_n(\aaa) = 0 \text{ if and only if } 
\langle f \cdot \Delta_n(\aaa), f \cdot \Delta_n(\aaa) \rangle = 0
\end{equation}
On the other hand, Lemma~\ref{bilinear-lemma} (2) implies that 
\begin{equation}
\langle f \cdot \Delta_n(\aaa), f \cdot \Delta_n(\aaa) \rangle = 
\pm \langle f, (f \cdot \Delta_n(\aaa)) \cdot \Delta_n(\aaa) \rangle = 0,
\end{equation}
where the second equality used $f \in W_n(\aaa)_j^{\perp}$.  This completes the proof of the Claim.

We proceed to prove that the map $\iota$ is surjective, which will prove the Theorem.  
Fix $0 \leq j \leq r$.  We prove that $W_n(\aaa)_{r - j}$ is contained in the image $\mathrm{Image}(\iota)$
of $\iota$.
By the definition of $W_n(\aaa)$, we have 
\begin{equation}
\label{known-equality-image-iota}
W_n(\aaa)_{r - j} = \QQ[\xx_n, {\bm \theta_n}]_j \cdot \Delta_n(\aaa) = 
\{ f \cdot \Delta_n(\aaa) \,:\, f \in \QQ[\xx_n, {\bm \theta_n}]_j \}.
\end{equation}
The desired containment $W_n(\aaa)_{r-j} \subseteq \mathrm{Image}(\iota)$ is equivalent to the 
seemingly stronger statement
\begin{equation}
\label{goal-equality-image-iota}
W_n(\aaa)_{r - j} = W_n(\aaa)_j \cdot \Delta_n(\aaa) = 
\{ f \cdot \Delta_n(\aaa) \,:\, f \in W_n(\aaa)_j \}.
\end{equation}
However, Lemma~\ref{perp-lemma} and our Claim give the direct sum decomposition
\begin{equation}
\label{main-direct-sum}
\QQ[\xx_n, {\bm \theta_n}]_j = W_n(\aaa)_j \oplus W_n(\aaa)^{\perp}_j = W_n(\aaa)_j \oplus I_n(\aaa)_j.
\end{equation}
Since $I_n(\aaa)_j$ annihilates $\Delta_n(\aaa)$, 
\eqref{known-equality-image-iota} and \eqref{goal-equality-image-iota} are equivalent.
\end{proof}

The map $\iota$ gives our desired duality immediately.

\begin{corollary}
\label{duality-corollary}
Let $0 \leq r \leq n$ and let $\aaa = (a_1, \dots, a_n) \in \ZZ_{\geq 0}^r$. The module 
$W_n(\aaa)$ has the duality property
\begin{equation}
\omega (\grFrob(W_n(\aaa); q, z)) = (\rev_q \circ \rev_z) \grFrob(W_n(\aaa); q, z).
\end{equation}
\end{corollary}

\begin{proof}
The automorphism $\iota$ of Theorem~\ref{duality-theorem} reverses 
both $x$-degree and $\theta$-degree, as well as twisting by the sign representation.
\end{proof}

In Corollary~\ref{duality-corollary} the operator $\rev_q$ acts on formal power series in 
$\QQ[[q, z, x_1, x_2, \dots ]]$ with finite $q$-degree by regarding them
as polynomials in $\QQ[[z, x_1, x_2, \dots ]][q]$.  A similar remark applies to $\rev_z$.
Corollary~\ref{duality-corollary} immediately shows that $W_n(\aaa)$ and $R_n(\aaa)$
are isomorphic as bigraded $\symm_n$-modules.

\begin{corollary}
\label{super-annihilator-twist}
For any $0 \leq r \leq n$ and any sequence $\aaa \in (\ZZ_{\geq 0})^r$ we have 
\begin{equation}
\grFrob(R_n(\aaa); q, z) =
\grFrob(W_n(\aaa); q, z)
\end{equation}
where $q$ tracks $x$-degree and $z$ tracks $\theta$-degree.
\end{corollary}

\begin{proof}
The same argument used to prove Proposition~\ref{annihilator-twist} shows that 
\begin{equation}
\grFrob(\QQ[\xx_n, {\bm \theta_n}]/I_n(\aaa) ; q, z) =
(\rev_q \circ \rev_z \circ \omega) \grFrob(W_n(\aaa); q, z).
\end{equation}
By Corollary~\ref{duality-corollary} the operator
$(\rev_q \circ \rev_z \circ \omega)$ leaves $\grFrob(W_n(\aaa);q,z)$ unchanged.
\end{proof}

Our duality gives us
another model for the 
quotient rings $R_{n,k}$ of \cite{HRS}.

\begin{corollary}
Let $r, k \geq 0$ with $n = k + r$.  Let $\aaa = (k-1, \dots , k-1) \in (\ZZ_{\geq 0})^n$ be a length $r$
sequence of $(k-1)$'s.  
Let $W_n(\aaa)_0$ be the subspace of $W_n(\aaa)$ of $\theta$-degree zero.  
Then $W_n(\aaa)_0 \cong R_{n,k}$ as singly graded $\symm_n$-modules.
\end{corollary}

\begin{proof}
Combine Theorem~\ref{delta-vandermonde-theorem} and Corollary~\ref{duality-corollary}.
\end{proof}

\subsection{Poincar\'e duality}
In this subsection we prove that the bigraded rings $R_n(\aaa)$ exhibit an algebraic
structure which is reminiscent of Poincar\'e duality and propose the problem
of finding a geometric explanation for this fact.

Consider Corollaries~\ref{duality-corollary} and \ref{super-annihilator-twist}
in the case $r = 0$, so that $\aaa = \varnothing$ is the empty sequence and $\Delta_n(\aaa) = \Delta_n$
is the classical Vandermonde.  In this case $R_{n,k} = R_n$ is the classical coinvariant algebra and these
corollaries give the classical result
\begin{equation}
\grFrob(R_n; q) = (\rev_q \circ \omega) \grFrob(R_n; q)
\end{equation}
which implies the palindromicity of the Hilbert series
\begin{equation}
\Hilb(R_n; q) = [n]!_q = (1+q)(1+q+q^2) \cdots (1+q + \cdots + q^{n-1}).
\end{equation}

The palindromicity of $\Hilb(R_n;q) = [n]!_{q}$ has a geometric interpretation.
Let $\mathcal{F \ell}_n$ be the variety of complete flags in $\CC^n$.
Borel \cite{Borel} proved that the rational cohomology ring $H^{\bullet}(\mathcal{F \ell}_n)$
can be presented as the coinvariant ring $R_n$.
Since $\mathcal{F \ell}_n$ is a smooth compact complex projective variety,
the top cohomology 
\begin{equation}
 H^{\mathrm{top}}(\mathcal{F \ell}_n) = H^{n(n-1)}(\mathcal{F \ell}_n) \cong \QQ
 \end{equation}
is a 1-dimensional vector space and for any $0 \leq d \leq n(n-1)$ the cup product
\begin{equation}
H^{d}(\mathcal{F \ell}_n) \otimes H^{n(n-1) - d}(\mathcal{F \ell}_n) \rightarrow \QQ
\end{equation}
is a perfect pairing by Poincar\'e duality.

The variety 
\begin{equation*}
X_{n,k} = \{ (\ell_1, \dots, \ell_n) \,:\, \ell_i \subseteq \CC^k, \, \dim(\ell_i) = 1, \, \ell_1 + \cdots + \ell_n = \CC^k \}
\end{equation*}
of spanning configurations of $n$ lines in $\CC^k$ 
introduced in \cite{PR} is not compact, and so does not satisfy the hypotheses of 
Poincar\'e duality.
Indeed,  
 the Hilbert series of its cohomology
\begin{equation}
\Hilb(H^{\bullet}(X_{n,k}); \sqrt{q}) = \Hilb(R_{n,k}; q) = \rev_q([k]! \cdot  \Stir_q(n,k))
\end{equation}
is not palindromic.
However,
$H^{\bullet}(X_{n,k}) = R_{n,k}$ is a $1$-dimensional slice of a $2$-dimensional self-dual object.

\begin{corollary}
\label{poincare-corollary}
Let $0 \leq r \leq n$ and let $\aaa = (a_1, \dots, a_r) \in (\ZZ)_{\geq 0}^r$.  
The top $x$-degree component is $R_n(\aaa)$ is $s := a_1 + \cdots + a_r + {k \choose 2}$
and the top $\theta$-degree component is $r$.
Write $R_n(\aaa)_{i,j}$ for the component of $R_n(\aaa)$ of $x$-degree $i$ and 
$\theta$-degree $j$.
The component $R_n(\aaa)_{r,s} \cong \QQ$ is a $1$-dimensional vector space.

For any $0 \leq i \leq s$ and any $0 \leq j \leq r$ the multiplication map
\begin{equation}
R_n(\aaa)_{i,j} \otimes R_n(\aaa)_{s-i,r-j} \rightarrow \QQ
\end{equation}
is a perfect pairing.
\end{corollary}

\begin{proof}
Let $\{f_1, \dots, f_m \}$ be a basis for $W_n(\aaa)_{i,j}$ and 
let $\{g_1, \dots, g_m \}$ be a basis for $W_n(\aaa)_{s-i,r-j}$
(by Theorem~\ref{duality-theorem}, these bases have the same size $m$).
The direct sum decomposition \eqref{main-direct-sum}
in the proof of
Theorem~\ref{duality-theorem} guarantees that 
$\{f_1, \dots, f_m\}$ descends to a basis of $R_n(\aaa)_{i,j}$ and
$\{g_1, \dots, g_m\}$ descends to a basis of $R_n(\aaa)_{s-i,r-j}$.

Let $A = (a_{p,q})_{1 \leq p, q \leq m}$ be the $m \times m$ rational matrix
whose entries are 
\begin{equation}
a_{p,q} := (f_p g_q) \cdot \Delta_n(\aaa).
\end{equation}
It is enough to verify that $A$ is nonsingular.
By Theorem~\ref{duality-theorem}, the set 
$\{ g_1 \cdot \Delta_n(\aaa), \dots, g_m \cdot \Delta_n(\aaa) \}$
descends to a basis of $R_n(\aaa)_{i,j}$.
The matrix element $a_{p,q}$ is equal to 
$\langle f_p, g_q \cdot \Delta_n(\aaa) \rangle$, so that $A$ is the Gram matrix
of a bilinear form on $R_n(\aaa)_{i,j}$ which (by Lemma~\ref{bilinear-lemma} (1))
is either positive definite or negative definite, and hence nonsingular.
\end{proof}

\begin{problem}
\label{poincare-problem}
Find
a geometric enhancement of $X_{n,k}$ which explains 
Corollary~\ref{poincare-corollary}.
\end{problem}

%We may consider the annihilator 
%\begin{equation}
%\ann_{\QQ[\xx_n, {\bm \theta_n}]} \Delta_n(\aaa)=
%\{ f \in \QQ[\xx_n, {\bm \theta_n}] \,:\, f \cdot \Delta_n(\aaa) = 0 \}
%\end{equation}
%of $\Delta_n(\aaa)$ inside $\QQ[\xx_n, {\bm \theta_n}]$.
%This is a (two-sided) ideal in $\QQ[\xx_n, {\bm \theta_n}]$ which is closed under the action
%of $\symm_n$.
%Proposition~\ref{annihilator-twist} has the following analogue whose proof is similar and omitted.

\subsection{Superspace coinvariants}
It is well-known that the ring $\QQ[\xx_n]^{\symm_n}$ of symmetric polynomials has algebraically 
independent generators given by the set $\{e_1, e_2, \dots, e_n \}$ of symmetric polynomials.
In general, a set of $n$ algebraically independent symmetric polynomials $\{f_1, f_2, \dots, f_n \}$ is a 
{\em fundamental system of invariants} if 
$\QQ[\xx_n]^{\symm_n} = \QQ[f_1, f_2, \dots, f_n]$.  A fundamental system of invariants other than
$\{e_1, e_2, \dots, e_n \}$ is the set of {\em power sums} $p_1, p_2, \dots, p_n$ where
$p_d = p_d(\xx_n) := x_1^d + \cdots + x_n^d$.

Recall that $\symm_n$ acts diagonally on superspace $\QQ[\xx_n, \bm{\theta}_n]$.  
The collection of invariant superpolynomials
$\QQ[\xx_n, {\bm \theta_n}]^{\symm_n}$ forms a subalgebra of $\QQ[\xx_n, {\bm \theta_n}]$.
Given a fundamental system of invariants $\{f_1, f_2, \dots, f_n \} \subseteq \QQ[\xx_n]$, Solomon
described a generating set of $\QQ[\xx_n, {\bm \theta_n}]^{\symm_n}$.  
The {\em differential} map $d: \QQ[\xx_n, \bm{\theta}_n] \rightarrow \QQ[\xx_n, \bm{\theta}_n]$
is defined by
\begin{equation}
df := \sum_{i = 1}^n \partial_i f  \times  \theta_i, \quad \text{for $f \in \QQ[\xx_n, \bm{\theta}_n]$.}
\end{equation}
Solomon proved \cite{Solomon} that 
$\QQ[\xx_n, \bm{\theta}_n]^{\symm_n}$ is generated as a $\QQ$-algebra by 
$\{ f_1, f_2, \dots, f_n, df_1, df_2, \dots, df_n \}$.
Extensions of various symmetric polynomial bases to superspace were studied in \cite{DLM}.

Let $\QQ[\xx_n, \bm{\theta}_n]^{\symm_n}_+$ denote the 
space of $\symm_n$-invariant superpolynomials
with vanishing constant term.  
The {\em superinvariant ideal} is the ideal 
\begin{equation}
SI_n := \langle \QQ[\xx_n, \bm{\theta}_n]^{\symm_n}_+ \rangle = 
\langle e_1, e_2, \dots, e_n, dp_1, dp_2, \dots, dp_n \rangle,
\end{equation}
where the second equality is justified by Solomon's result \cite{Solomon} and the 
equality $\QQ[e_1, \dots, e_n] = \QQ[p_1, \dots, p_n]$.
The {\em supercoinvariant algebra} is the quotient
\begin{equation}
SR_n :=  \QQ[\xx_n, \bm{\theta}_n]/ \langle \QQ[\xx_n, \bm{\theta}_n]^{\symm_n}_+ \rangle.
\end{equation}
The ring $SR_n$ is a bigraded $\symm_n$-module.

The following conjectural expression for the bigraded Frobenius image of $SR_n$ was 
obtained in the algebraic combinatorics seminar at  the Fields Institute; it is a special case of 
the conjecture of Mike
Zabrocki  in \cite{Zabrocki}:
\begin{equation}
\label{zabrocki-conjecture}
\grFrob(SR_n; q,z) = \sum_{k = 1}^n z^{n-k} \cdot \Delta'_{e_{k-1}} e_n \mid_{t = 0}
\end{equation}
where  $q$ tracks $x$-degree and $z$ tracks $\theta$-degree.
By Theorem~\ref{delta-vandermonde-theorem},
Equation~\eqref{zabrocki-conjecture} is equivalent to the statement that  
the $\theta$-homogeneous piece of $SR_n$ of $\theta$-degree $n-k$ is isomorphic to
$V_n(\aaa)$ as a (singly) graded $\symm_n$-module, where $\aaa = (k-1, \dots, k-1)$ is a length
$n-k$ sequence of $(k-1)$'s.
%Equation~\eqref{zabrocki-conjecture} predicts that the vector space dimension of $SR_n$ is
%$\sum_{k = 1}^n k! \cdot \Stir(n,k)$, the total number of ordered set partitions of $[n]$.
%N. Bergeron, Machacek and Zabrocki (personal communication) proved
%$\dim SR_n \geq \sum_{k = 1}^n k! \cdot \Stir(n,k)$.  
By Theorem~\ref{delta-vandermonde-theorem} we have the following potential road to proving
Equation~\eqref{zabrocki-conjecture}.

\begin{proposition}
\label{zabrocki-reformulation}
Suppose that, for each $1 \leq k \leq n$, the following condition holds:
\begin{quote}
the canonical projection from $V_n(\aaa) \subseteq \QQ[\xx_n, \bm{\theta}_n]$,
where $\aaa = (k-1, \dots, k-1)$ is a length $n-k$ sequence of $(k-1)$'s, to 
the $\theta$-homogeneous piece of $SR_n$ of $\theta$-degree $n-k$ is an isomorphism of vector spaces.
\end{quote}
Then Equation~\eqref{zabrocki-conjecture} is true.
\end{proposition}

We have been unable to use Vandermondes to analyze the supercoinvariant algebra 
$SR_n$ directly, but we have the following result 
describing elements of the annihilator of $\Delta_n(\aaa)$ in terms of
the superinvariant ideal $SI_n$ in the case where $\aaa = (k-1, \dots, k-1)$ is a constant
sequence of $(k-1)$'s of length $n-k$.

\begin{proposition}
\label{rho-ideal-kills}
Let $0 \leq r \leq n$ and let $k = n - r$.  Let $\aaa \in (\ZZ_{\geq 0})^r$ be the constant sequence
$\aaa = (k-1, k-1, \dots, k-1)$ of length $r$.  Each of the superpolynomials
\begin{equation}
x_1^k, x_2^k, \dots, x_n^k, \quad
e_n, e_{n-1}, \dots, e_{n-k+1}, \quad
dp_1, dp_2, \dots, dp_n
\end{equation}
annihilates $\Delta_n(\aaa)$.
\end{proposition}

The ideal generated by the superpolynomials
appearing in Proposition~\ref{rho-ideal-kills} has generators
similar to the superinvariant ideal
$SI_n = \langle e_1, \dots, e_n, dp_1, \dots, dp_n \rangle$, but without the low degree elementary symmetric
polynomials $e_1, e_2, \dots, e_{n-k}$ and with the variable powers $x_1^k, x_2^k, \dots, x_n^k$.
Indeed, these ideals are not equal. Despite this,
we hope that the similarity between these ideals will assist in the proof of Zabrocki's 
conjecture \eqref{zabrocki-conjecture}.

\begin{proof}
We have $x_1^k, \dots, x_n^k \in \ann_{\QQ[\xx_n]} \Delta_n(\aaa)$ because no
$x$-variable in $\Delta_n(\aaa)$ has exponent $\geq k$.  
By Lemma~\ref{sign-reversing-lemma} we also have 
$e_n, e_{n-1}, \dots, e_{n-k+1} \in \ann_{\QQ[\xx_n]} \Delta_n(\aaa)$.

Let $1 \leq j \leq n$.
By \eqref{permutation-derivative-relations}
$dp_j$ commutes with the action of $\symm_n$, and hence the action of $\varepsilon_n$.
It follows that 
\begin{equation}
\label{step-one}
dp_j \cdot \Delta_n(\aaa)  = \varepsilon_n \cdot dp_j \cdot 
\left[ x_1^{k-1}   \cdots x_{r}^{k-1}
x_{r+1}^{k-1} \cdots x_{n-1}^1  x_n^0 \cdot \theta_1 \theta_2 \cdots \theta_r \right].
\end{equation}
A direct computation gives $dp_j \cdot \left[ x_1^{k-1}   \cdots x_{r}^{k-1}
x_{r+1}^{k-1} \cdots x_{n-1}^1  x_n^0 \cdot \theta_1 \theta_2 \cdots \theta_r \right] = 0$
if $j > k$ and 
\begin{equation}
\label{step-two}
dp_j \cdot 
\left[ x_1^{k-1}   \cdots x_{r}^{k-1}
x_{r+1}^{k-1} \cdots  x_n^0 \cdot \theta_1  \cdots \theta_r \right] \doteq
\sum_{i = 1}^r (-1)^{i-1} x_1^{k-1} \cdots x_i^{k-j} \cdots x_r^{k-1} x_{r+1}^{k-1} \cdots x_n^0 \theta_1 
\cdots \widehat{\theta_i} \cdots \theta_r 
\end{equation}
if $j \leq k$.
In term $i$ in the sum on the right-hand-side of Equation~\eqref{step-two}, the exponents
of $x_i$ and $x_{n-k+j}$ coincide.  Since neither $\theta_i$ nor $\theta_{n-k+j}$ appear in this term,
this term is annihilated by the application of $\varepsilon_n$. 
We conclude that Equation~\eqref{step-two} itself is annihilated by $\varepsilon_n$, so that 
Equation~\eqref{step-one} equals 0.
\end{proof}

\section{Conclusion}
\label{Conclusion}
\subsection{A conjecture on Tanisaki quotients}
In this paper we defined a graded $\symm_n$-module $V_n(\aaa)$ for any nonnegative integer
sequence $\aaa$ of length $\leq n$.
Theorem~\ref{delta-vandermonde-theorem},
Theorem~\ref{higher-constant-sequences}, and
Proposition~\ref{zero-tanisaki}
calculate the graded module structure of $V_n(\aaa)$ for certain constant sequences $\aaa$.
It is natural to ask what $V_n(\aaa)$ looks like for general sequences $\aaa$.
While we do not have a full conjecture in this direction, computational evidence suggests a relationship
between $V_n(\aaa)$ and the Tanisaki quotients $R_{\lambda}$.

More precisely, let $\leq$ be the componentwise partial order on length $r$ sequences of 
nonnegative integers.  Given a length $r$ sequence $\aaa \in (\ZZ_{\geq 0})^r$ and $n \geq r$, define the 
graded $\symm_n$-modules
\begin{equation}
V^{\leq}_n(\aaa) := \sum_{\mathbf{b} \leq \aaa} V_n(\mathbf{b}), \quad
V^{<}_n(\aaa) := \sum_{\mathbf{b} < \aaa} V_n(\mathbf{b}), \quad \text{and } 
V^=_n(\aaa) := V^{\leq}_n(\aaa)/V^<_n(\aaa).
\end{equation}
It can be checked that 
\begin{equation*}
\begin{cases}
(\rev_q \circ \omega) \grFrob(V^=_4(0,0); q) = \grFrob(R_{(3,1)}; q), \\
(\rev_q \circ \omega) \grFrob(V^=_4(1,0); q) = \grFrob(R_{(3,1)}; q), \text{ and } \\\
(\rev_q \circ \omega) \grFrob(V^=_4(1,1); q) = \grFrob(R_{(2,2)}; q). 
\end{cases}
\end{equation*}

\begin{conjecture}
\label{tanisaki-conjecture}
Let $r \leq n$ be nonnegative integers with $k = n-r$ and let $\aaa \in (\ZZ_{\geq 0})^r$.
There exists a partition $\lambda \vdash n$ with $k$ parts such that 
\begin{equation}
(\rev_q \circ \omega) \grFrob(V^=_n(\aaa); q) = \grFrob(R_{\lambda}; q).
\end{equation}
Equivalently, if $Q'_{\lambda}(X;q)$ is the Hall-Littlewood $Q'$-function,
we have
\begin{equation}
\label{second-formulation}
\omega [\grFrob(V^=_n(\aaa); q)] \propto 
Q'_{\lambda}(X;q),
\end{equation}
where $\propto$ denotes equality up to a power of $q$.
\end{conjecture}

Proposition~\ref{zero-tanisaki} proves Conjecture~\ref{tanisaki-conjecture} when $\aaa$ is a 
zero sequence and $\lambda = (r+1, 1, \dots, 1) \vdash n$.
By Proposition~\ref{annihilator-twist} and  \cite[Thm. 6.14]{HRS} 
if $\aaa = (k-1, \dots, k-1)$ is a length $r$ sequence of $(k-1)$'s we have
\begin{equation}
\label{hl-identity}
\grFrob(V_n(\aaa); q) = 
\sum_{\substack{\lambda \vdash n \\ \ell(\lambda) = k}}
q^{\sum (i-1)(\lambda_i - 1)}
{k \brack m_1(\lambda), \dots, m_n(\lambda)}_q \omega Q'_{\lambda}(X;q).
\end{equation}
Conjecture~\ref{tanisaki-conjecture} can be thought of as giving a filtration on 
$V_n(\aaa)$ which is compatible with Equation~\eqref{hl-identity}.
We do not have a conjecture for how to produce $\lambda$ from $\aaa$ in general.

The generalized coinvariant ring $R_{n,k}$ of \cite{HRS} and the positroid quotient $S_n$ of \cite{BRT}
have graded Frobenius images which are (up to $q$-reversal) positive in the 
$Q'$-basis of symmetric functions.  In \cite{RW2} the authors defined a quotient of the polynomial ring
$\QQ[\xx_n]$ corresponding to hook Schur-delta operator images
$\Delta_{s_{(r, 1^{n-1})}} e_n \mid_{t = 0}$ whose graded Frobenius image is also (up to $q$-reversal) $Q'$-positive.
Haglund, Rhoades, and Shimozono  \cite{HRS2} gave a manifestly positive $Q'$-expansion of 
$\Delta'_{s_{\lambda}} e_n \mid_{t = 0}$, where $s_{\lambda}$ is any Schur function (up to $\omega$).
It may be interesting to use superspace to build modules for the symmetric functions appearing 
in \cite{RW2} and \cite{HRS2}.

\subsection{Additional sets of variables, $\Delta'_{e_{k-1}} e_n$, and beyond}
Zabrocki \cite{Zabrocki} conjectured an extension of \eqref{zabrocki-conjecture} to more sets of variables.
Let  $\QQ[\xx_n, \yy_n, \bm{\theta}_n]$ be the $\QQ$-algebra
with $2n$ commuting variables $x_1, \dots, x_n, y_1, \dots, y_n$ and
$n$ anticommuting variables $\theta_1, \dots, \theta_n$
(where any two variables of different species commute). Zabrocki \cite{Zabrocki}
verified that 
\begin{equation}
\label{conjectural-delta-quotient}
\grFrob(\QQ[\xx_n, \yy_n, \bm{\theta}_n]/ \langle \QQ[\xx_n, \yy_n, \bm{\theta}_n]^{\symm_n}_+ \rangle; q,t,z) 
= \sum_{k = 1}^n z^{n-k} \cdot \Delta'_{e_{k-1}} e_n
\end{equation}
for $n \leq 6$.
Here $q$ tracks $x$-degree, $t$ tracks $y$-degree, and $z$ tracks $\theta$-degree.
Thus, the quotient 
$\QQ[\xx_n, \yy_n, \bm{\theta}_n]/ \langle \QQ[\xx_n, \yy_n, \bm{\theta}_n]^{\symm_n}_+ \rangle$ gives 
a conjectural representation theoretic model for $\Delta'_{e_{k-1}} e_n$
in antisymmetric degree $n-k$.

Superspace Vandermondes can be used to give another conjectural representation theoretic model for
$\Delta'_{e_{k-1}} e_n$.  To describe this, we need the {\em polarization operators}
on $\QQ[\xx_n, \yy_n, \bm{\theta}_n]$.
For $j \geq 1$, the $j^{th}$ {\em polarization operator (from the $x$-variables to the $y$-variables)}
on $\QQ[\xx_n, \yy_n, \bm{\theta}_n]$ is the operator
\begin{equation}
p_{x \rightarrow y}^{(j)} := y_1 \partial_{x_1}^j + y_2 \partial_{x_2}^j + \cdots + y_n \partial_{x_n}^j.
\end{equation}

\begin{defn}
Let $n = k + r$ and let $\aaa = (a_1, \dots, a_r) \in (\ZZ_{\geq 0})^r$.
Let $\mathcal{V}_n(\aaa)$ be the smallest subspace of  $\QQ[\xx_n, \yy_n, \bm{\theta}_n]$ such that
\begin{itemize}
\item $\mathcal{V}_n(\aaa)$ contains the $\aaa$-superspace Vandermonde 
\begin{equation*}
\Delta_n(\aaa) = \varepsilon_n \cdot (x_1^{a_1} \cdots x_r^{a_r} x_{r+1}^{k-1} \cdots x_{n-1}^1 x_n^0 
\theta_1 \cdots \theta_r)
\end{equation*}
in the $x$-variables and $\theta$-variables,
\item  $\mathcal{V}_n(\aaa)$ is closed under all partial derivatives $\partial_{x_i}$ and $\partial_{y_i}$ in 
commuting variables, and
\item $\mathcal{V}_n(\aaa)$ is closed under all polarization operators $p_{x \rightarrow y}^{(j)}$ for $j \geq 1$.
\end{itemize}
\end{defn}

The space $\mathcal{V}_n(\aaa)$ has fixed $\theta$-degree $r$. By considering the $x$-degree and $y$-degree
separately, we view $\mathcal{V}_n(\aaa)$ as a doubly graded $\symm_n$-module.
The space $\mathcal{V}_n(\aaa)$ specializes to $V_n(\aaa)$ when the $y$-variables are set to zero.

\begin{conjecture}
\label{double-frobenius}
Let $n = k+r$ and
let $\aaa = (k-1, k-1, \dots, k-1)$ be a length $r$ sequence of $(k-1)$'s.  
Then
\begin{equation}
\grFrob(\mathcal{V}_n(\aaa); q, t) = \Delta'_{e_{k-1}} e_n.
\end{equation}
\end{conjecture}

Conjecture~\ref{double-frobenius} has been checked by computer for $n \leq 4$.
Theorem~\ref{delta-vandermonde-theorem} proves Conjecture~\ref{double-frobenius} in the case $t = 0$;
Zabrocki's conjecture \eqref{conjectural-delta-quotient} is open even in the case $t = 0$.
Just as we hope that $V_n(\aaa)$ will lead to a better understanding of the supercoinvariant 
algebra 
$SR_n = \QQ[\xx_n, \bm{\theta}_n]/ \langle \QQ[\xx_n, \bm{\theta}_n]^{\symm_n}_+ \rangle$,
we hope that $\mathcal{V}_n(\aaa)$ will help in understanding  the quotient
$\QQ[\xx_n, \yy_n, \bm{\theta}_n]/ \langle \QQ[\xx_n, \yy_n, \bm{\theta}_n]^{\symm_n}_+ \rangle$
appearing in \eqref{conjectural-delta-quotient}.

Given any  vector $\aaa = (a_1, \dots, a_r) \in (\ZZ_{\geq 0})^r$,
we have a symmetric function $\grFrob(\mathcal{V}_n(\aaa);q,t)$.
It might be interesting to study the combinatorics of these symmetric functions when $\aaa$ 
is a vector other than $(k-1, \dots, k-1)$.

There has been a significant amount of interest in extensions of the diagonal 
coinvariant ring to $> 2$ species of $n$ commuting variables (see \cite{BergeronMulti}).
Let us remark that we may extend our modules $\mathcal{V}_n(\aaa)$ to any number of 
species of commuting and skew-commuting variables.

Let $S(n,c,s)$ be the $\QQ$-algebra generated by $c$ species of $n$ {\bf c}ommuting variables
\begin{equation*}
x_1^{(1)}, x_2^{(1)}, \dots, x_n^{(1)}, \quad
x_1^{(2)}, x_2^{(2)}, \dots, x_n^{(2)}, \quad \dots \quad, \text{ and }
x_1^{(c)}, x_2^{(c)}, \dots, x_n^{(c)},
\end{equation*}
and $s$ species of $n$ {\bf s}kew-commuting variables 
\begin{equation*}
\theta_1^{(1)}, \theta_2^{(1)}, \dots, \theta_n^{(1)}, \quad
\theta_1^{(2)}, \theta_2^{(2)}, \dots, \theta_n^{(2)}, \quad \dots \quad, \text{ and }
\theta_1^{(s)}, \theta_2^{(s)}, \dots, \theta_n^{(s)},
\end{equation*}
where any two variables drawn from different species commute.
The ring $S(n,c,s)$ is a multigraded $\QQ$-algebra with $c$ kinds of commutative grading and
$s$ kinds of skew-commutative grading.

The ring $S(n,c,s)$ carries a `diagonal' action of $\symm_n$ by simultaneous subscript permutation.
We may also act on $S(n,c,s)$ by any partial derivative $\partial/\partial x_i^{(j)}$ or 
$\partial/\partial \theta_i^{(j)}$ 
with respect to any commuting or skew-commuting variable.
We may also polarize between any two species of commutating variables $1 \leq i, i' \leq c$
(at polarization parameter $j$)
by the operator
\begin{equation}
p^{(j)}_{i \rightarrow i'} := x_1^{(i')} (\partial/\partial x_1^{(i)})^j +
x_2^{(i')} (\partial/\partial x_2^{(i)})^j  + \cdots + x_n^{(i')} (\partial/\partial x_n^{(i)})^j 
\end{equation}
and between any two species of skew-commuting variables $1 \leq i, i' \leq s$ by the operator
\begin{equation}
q^{(j)}_{i \rightarrow i'} := \theta_1^{(i')} (\partial/\partial \theta_1^{(i)}) +
\theta_2^{(i')} (\partial/\partial \theta_2^{(i)})  + \cdots + \theta_n^{(i')} (\partial/\partial \theta_n^{(i)}).
\end{equation}

Given $r \leq n$ and sequence $\aaa = (a_1, \dots, a_r) \in (\ZZ_{\geq 0})^r$, we may define 
$\mathbb{V}_n(\aaa,c,s)$ to be the smallest $\QQ$-linear subspace of $S(n,c,s)$ 
containing the $\aaa$-superspace Vandermonde $\Delta_n(\aaa)$ in the 
$x^{(1)}$ and $\theta^{(1)}$-variables which is closed under all possible partial differentiation and 
polarization operators.
We have $\mathbb{V}_n(\aaa,1,1) = V_n(\aaa)$ and $\mathbb{V}_n(\aaa,2,1) = \mathcal{V}_n(\aaa)$.

The vector space $\mathbb{V}_n(\aaa,c,s)$ is a multigraded $\symm_n$-module. 
Its isomorphism type is encoded in a symmetric function
\begin{equation}
\grFrob(\mathbb{V}_n(\aaa,c,s); q_1, q_2, \dots, q_c, z_1, z_2, \dots, z_s)
\end{equation}
Following the work
of F. Bergeron \cite{Bergeron}, it may be interesting to study this symmetric function
as $c, s \rightarrow \infty$.

\subsection{A conjectural Lefschetz property of $W_n(\aaa)$}
For $r \leq n$ and $\aaa \in (\ZZ_{\geq 0})^r$, we defined a doubly graded $\symm_n$-module
$W_n(\aaa)$. 

\begin{problem}
Find the doubly graded Frobenius image $\grFrob(W_n(\aaa); q, z)$.
\end{problem}

The $z^0$-coefficient of $\grFrob(W_n(\aaa); q, z)$ gives the graded isomorphism type of $R_{n,k}$.
The $z^r$-coefficient of $\grFrob(W_n(\aaa); q, z)$ is the reversed and sign-twisted version of 
the $z^0$-coefficient.  The authors do not have a conjecture for the intermediate powers of $z$.
Indeed, we do not even know the vector space dimension $\dim W_n(\aaa)$.

Let $r \leq n$ and $\aaa \in (\ZZ_{\geq 0})^r$.
We close with a conjecture on the bigraded Hilbert series 
\begin{equation}
\label{double-hilbert-definition}
\Hilb(R_n(\aaa); q, z) := \sum_{i,j} \dim  R_n(\aaa)_{i,j} \cdot q^i z^j
\end{equation}
of the doubly graded ring $R_n(\aaa)$.
Here $R_n(\aaa)_{i,j}$ is the homogeneous piece of $R_n(\aaa)$ of $x$-degree $i$ and $\theta$-degree $j$.
By Proposition~\ref{super-annihilator-twist} the polynomial \eqref{double-hilbert-definition} 
is unchanged if we replace $R_n(\aaa)$ by the doubly graded vector space $W_n(\aaa)$.

We may display the bivariate polynomial \eqref{double-hilbert-definition} as a matrix of coefficients.
The case $n = 5, \aaa = (2,2)$ is shown below, with column indices recording $x$-degree and row
indices recording $\theta$-degree.
\begin{equation*}
\begin{pmatrix}
1 & 5 & 15 & 29 & 39 & 35 & 20 & 6 \\
4 & 19 & 50 & 77 & 77 & 50 & 19 & 4 \\
6 & 20 & 35 & 39 & 29 & 15 & 5 & 1
\end{pmatrix}
\end{equation*}
As guaranteed by Theorem~\ref{duality-theorem}, this matrix is symmetric under $180^{\circ}$ rotation.
Recall that an integer sequence $(c_1, c_2, \dots, c_m)$ is {\em unimodal} if there is some $i$ with
$c_1 \leq \cdots \leq c_i \geq c_{i+1} \geq \cdots \geq c_m$.

\begin{conjecture}
\label{unimodality-conjecture}
For any $r \leq n$ and $\aaa \in (\ZZ_{\geq 0})^r$, the matrix of coefficients of
$\Hilb(R_n(\aaa); q, z)$ has unimodal rows and columns.
\end{conjecture}

When $\aaa = \varnothing$, Conjecture~\ref{unimodality-conjecture} follows from the Hard Lefschetz
Theorem.  In this case, we have the geometric interpretation 
$R_n(\varnothing) = R_n = H^{\bullet}(\mathcal{F \ell}_n)$ 
of $R_n(\varnothing)$ as  the (singly-graded) cohomology of the K\"ahler manifold $\mathcal{F \ell}_n$.
Recall that $n(n-1)$ is the top degree of the cohomology ring
$H^{\bullet}(\mathcal{F \ell}_n)$.
The Hard Lefschetz Theorem states that there is an element $\ell \in H^2(\mathcal{F \ell}_n)$ 
such that for all $d \leq {n \choose 2}/2$ the multiplication map
\begin{equation}
\ell^{{n \choose 2} - 2d} \times (-): H^{2d}(\mathcal{F \ell}_n) \rightarrow H^{n(n-1) - 2d}(\mathcal{F \ell}_n)
\end{equation}
is an isomorphism of vector spaces.
Any element $\ell \in H^2(\mathcal{F \ell}_n)$ with this property is called a {\em (strong) Lefschetz element}.
In terms of the presentation
 $H^{\bullet}(\mathcal{F \ell}_n) = R_n = \QQ[x_1, \dots, x_n]/\langle e_1, \dots, e_n \rangle$, we may
 represent any element $\ell \in H^2(\mathcal{F \ell}_n)$ as a $\QQ$-linear combination 
 $c_1 x_1 + \cdots + c_n x_n$ of the variables $x_1, \dots, x_n$.  The element $\ell$ is Lefschetz if and only if 
$c_i \neq c_j$ for all $i \neq j$ \cite{MNW}.

Conjecture~\ref{unimodality-conjecture} would be best proven by a doubly graded version of the 
Hard Lefschetz Theorem.
For the symmetric grading, one could hope that multiplication by an appropriate
 linear form $\ell = c_1 x_1 + \cdots + c_n x_n$ with $c_i \neq c_j$ for $i \neq j$ would 
be surjective or injective depending on the relative sizes of the entries in a row of 
$\Hilb(R_n(\aaa);q,z)$.
On the other hand, if $\tau = c_1 \theta_1 + \cdots + c_n \theta_n$ is {\em any} $\QQ$-linear combination
of the $\theta$-variables, we have $\tau^2 = 0$, so we would need a new model for
the antisymmetric part of a Lefschetz element.

Ideally, the unimodality of Conjecture~\ref{unimodality-conjecture} would be explained by 
the geometry of objects with algebraic invariants given by superspace quotients.
We leave this project for future work.

\section{Acknowledgements}
\label{Acknowledgements}

B. Rhoades is grateful to the organizers and hosts of
`Representation Theory Connections to $(q,t)$-Combinatorics' at BIRS in January 2019
where this project was initiated.
B. Rhoades was partially supported by NSF Grant DMS-1500838.
The authors are grateful to 
Nantel Bergeron, Sara Billey, Jeff Rabin,
Josh Swanson, and Mike Zabrocki for helpful discussions.
In particular, the authors thank Mike Zabrocki for suggesting the use 
of polarization operators.

\end{document}